\documentclass[12pt, letterpaper]{article}
\usepackage{amssymb}
\usepackage{amsmath}
\usepackage[top=1.2in, bottom=1.4in, left=1in, right=1in]{geometry}
\usepackage{parskip}
 \usepackage[T1]{fontenc}
\usepackage{lmodern}
 \usepackage{centernot}
\usepackage{epsfig}
\usepackage{amsmath,amsthm,amsfonts,amssymb,cite,amscd}
\usepackage{indentfirst}
\usepackage{mathrsfs}
\usepackage{amsthm}
\usepackage{tikz}
\usepackage{mathtools}
\usepackage{calc}
\usepackage[symbol]{footmisc}
\usetikzlibrary{patterns,arrows,decorations.pathreplacing}
\usetikzlibrary{fadings}
\usepackage{pgfplots}
\usepackage{parskip}
\newtheorem{theorem}{Theorem}

\newtheorem{lemma}[theorem]{Lemma}

\newtheorem{definition}{Definition}

\newtheorem{conjecture}{Conjecture}

\DeclareMathOperator{\ex}{ex}
\DeclareMathOperator{\otherwise}{otherwise}
\newtheorem{proposition}[theorem]{Proposition}

\begin{document}
\baselineskip=0.30in

\begin{center}
{\LARGE  \bf \textbf{The Tur\'an number of the square of a path}}\\
\vskip 0.3cm
\author\centerline{Chuanqi Xiao \footnote[1]{email:chuanqixm@gmail.com}\\
Central European University, Budapest, Hungary\\
 Gyula O.H. Katona
\footnote[2]{email:katona.gyula.oh@renyi.hu}\\
 MTA R\'{e}nyi Institute, Budapest, Hungary\\
  Jimeng Xiao\\
 Northwestern Polytechnical University, Xi$^{'}$an, China\\
 MTA R\'{e}nyi Institute, Budapest, Hungary}\\
 Oscar Zamora\\
 Central European University, Budapest, Hungary\\
 Universidad de Costa Rica, San Jos\'{e}, Costa Rica\\

\end{center}

{\bf Abstract}\ \
The Tur\'an number of a graph $H$, $\ex(n,H)$, is the maximum number of edges in a
graph on $n$ vertices which does not have $H$ as a subgraph. Let $P_{k}$ be the path with $k$ vertices, the square $P^{2}_{k}$ of $P_{k}$ is obtained by joining the pairs of
vertices with distance one or two in $P_{k}$. The powerful theorem of Erd\H{o}s, Stone and Simonovits determines the asymptotic
behavior of $\ex(n,P^{2}_{k})$.
In the present paper, we determine the exact value of $\ex(n,P^{2}_{5})$ and $\ex(n,P^{2}_{6})$ and pose a conjecture for the exact value of $\ex(n,P^{2}_{k})$.\\
{\bf Keywords}\ \ Tur\'an number, Extremal graphs, Square of a path.
\section{ Introduction}
\noindent In this paper, all graphs considered are undirected, finite and contain neither loops nor multiple edges. Let $G$ be such a graph, the vertex set of $G$ is denoted by $V(G)$, the edge set of $G$ by $E(G)$, and the number of edges in $G$ by $e(G)$. We denote the degree of a vertex $v$ by $d(v)$, the minimum degree in graph $G$ by
$\delta(G)$, the neighborhood of $v$ by $N(v)$ and the chromatic
number of graph $G$ by $\chi(G)$.

The following graphs will be studied in the present paper. Let $P_{k}$ be the path with $k$ vertices, the square $P^{2}_{k}$ of $P_{k}$ is obtained by joining the pairs of vertices with distance one or two in $P_{k}$, see Figure \ref{P_k}. The Tur\'an number of a graph $H$, $\ex(n,H)$, is the maximum number of edges in a graph on $n$ vertices which does not have $H$ as a subgraph. Our goal in this paper is to study $\ex(n,P^{2}_{k})$ and the extremal graphs for $P^2_k$. The Erd\H{o}s-Stone-Simonovits
Theorem \cite{EDR1,EDR2} asymptotically determines $\ex(n,H)$ for all non-bipartite graphs $H$:\\
\centerline{$\ex(n,H)=(1-\frac{1}{\chi(H)-1})\binom{n}{2}+o(n^{2})$.}
Since $\chi(P^{2}_{k})=3$, $k\geq3$, we have $\ex(n,P^{2}_{k})=\frac{n^{2}}{4}+o(n^{2})$. Yet, it still remains interesting to determine the exact value of $\ex(n,P^{2}_{k})$.
\begin{figure}[ht]
\centering
\begin{tikzpicture}

\draw[blue,thick] (0,0) arc(180:0:1cm and 1cm) (1,0) arc(180:0:1cm and 1cm) (2,0) arc(180:0:1cm and 1cm) (3,0) arc(180:0:1cm and 1cm) (6,0) arc(180:0:1cm and 1cm) (7,0) arc(180:0:1cm and 1cm) (6,0) arc(180:0:1cm and 1cm);

\filldraw (0,0) circle (2pt) node[below]{$v_1$} -- (1,0) circle (2pt) node[below]{$v_2$} -- (2,0) circle (2pt) node[below]{$v_3$} -- (3,0) circle (2pt) node[below]{$v_4$}--(4,0) circle (2pt)node[below]{$v_5$} (7,0) circle (2pt) node[below]{$v_{k-2}$} -- (8,0) circle (2pt) node[below]{$v_{k-1}$} -- (9,0) circle (2pt) node[below]{$v_k$};

\draw[dashed] (4,0) -- (7,0);
\end{tikzpicture}
\caption{Graph $P^{2}_{k}$}
\label{P_k}
\end{figure}
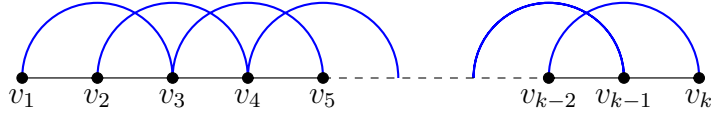

The very first result of extremal graph theory gave the value of $\ex(n,P^{2}_{3})$.
\begin{theorem}[\textbf{Mantel}\cite{MAN}] The maximum number of edges in an $n$-vertex triangle-free graph is $\lfloor \frac{n^{2}}{4}\rfloor$, that is $\ex(n, P^{2}_{3})=\lfloor \frac{n^2}{4}\rfloor$. Furthermore, the only triangle-free graph with $\lfloor \frac{n^2}{4}\rfloor$ edges is the complete bipartite graph $K_{\lfloor \frac{n}{2}\rfloor, \lceil\frac{n}{2} \rceil}$.
\end{theorem}
The case $k=4$ was solved by Dirac in a more general context.
\begin{theorem}[\textbf{Dirac}\cite{DIR}]\label{DIR} The maximum number of edges in an $n$-vertex $P^{2}_{4}$-free graph is $\lfloor\frac{n^{2}}{4}\rfloor$, that is $\ex(n, P^{2}_{4})= \lfloor\frac{n^{2}}{4}\rfloor$, $(n\geq4)$. Furthermore, when $n\geq5$, the only extremal graph is the complete bipartite graph $K_{\lfloor \frac{n}{2}\rfloor, \lceil\frac{n}{2} \rceil}$.
\end{theorem}
For $k=5$, our results are given in the next two theorems, where we separate the result for the Tur\'an number and the extremal graphs for $P^2_5$.
\begin{theorem}\label{P5}
The maximum number of edges in an $n$-vertex $P^{2}_{5}$-free graph is $\lfloor\frac{n^{2}+n}{4}\rfloor$, that is $\ex(n, P^{2}_{5})= \lfloor\frac{n^{2}+n}{4}\rfloor$, $(n\geq5)$.
\end{theorem}
\begin{definition}
 Let $E^{i}_{n}$ denote a graph obtained from a complete bipartite graph $K_{i,n-i}$ plus a maximum matching in the class which has $i$ vertices, see Figure \ref{E^{i}_{n}}.
\end{definition}
\begin{theorem}\label{EP5}
Let $n$ be a natural number, when $n=5$, the extremal graphs for $P_5^2$ are $E^2_5$, $E^3_5$ and $G_0$, where $G_0$ is obtained from a $K_4$ plus a pendent edge. When $n\geq6$, if $n\equiv 1,2~(\bmod~4)$, the extremal graphs for $P_5^2$ are $E^{\lceil\frac{n}{2}\rceil}_{n}$ and $E^{\lfloor\frac{n}{2}\rfloor}_{n}$, otherwise, the extremal graph for $P_5^2$ is $E^{\lceil\frac{n}{2}\rceil}_{n}$.
 \end{theorem}
  \begin{figure}[ht]
\centering
\begin{tikzpicture}

\foreach \x in {-.7,-1.4,.7,1.4,2.1,2.8,3.5}{
\foreach \y in {-.7,-1.4,.7,1.4,2.1,2.8,3.5}{

\filldraw (\y,0)  -- (\x,-2) ;

}
}

\foreach \x in {-.7,-1.4,.7,1.4,2.1,2.8,3.5}{

\filldraw (\x,0) circle (2pt)  (\x,-2) circle (2pt);

}

\draw (0,0) node{$\dots$} (0,-2) node{$\dots$};


\draw (-.7,0) -- (-1.4,0)  (.7,0) -- (1.4,0)  (2.1,0) -- (2.8,0);


\draw  (3.9,-2) arc (0:360:2.85cm and .5cm) node[align=center,right]{$Y$};

\draw (-2.5,0) node[right]{$i$};

\draw (-3,-2) node[right]{$n-i$};

\draw (3.9,-1) node[right]{$K_{i,n-i}$};

\draw (3.9,0) arc (0:360:2.85cm and .5cm) node[align=center,right]{$X$};

\end{tikzpicture}
\caption{Graph $E^{i}_{n}$}
\label{E^{i}_{n}}
\end{figure}
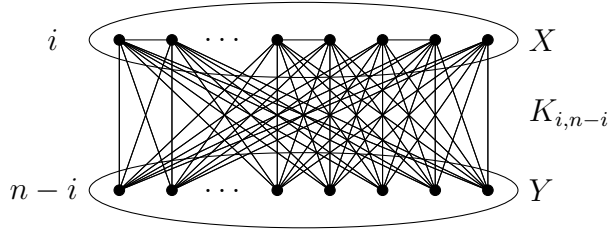
\begin{definition}
Let $T$ denote the flattened tetrahedron, see $T$ in Figure \ref{T_n}.
\end{definition}
Although the determination of $\ex(n,T)$ is not within the main lines of our paper, we need the exact value of $\ex(n,T)$ in order to determine $\ex(n,P^2_6)$.
\begin{theorem}\label{T}
The maximum number of edges in an $n$-vertex $T$-free graph ($n\neq 5$) is,
$$ \ex(n,T)=\left\{
\begin{aligned}
&\left\lfloor\frac{n^{2}}{4}\right\rfloor+\left\lfloor\frac{n}{2}\right\rfloor,& n\not\equiv2~(\bmod~4), \\
&\frac{n^{2}}{4}+\frac{n}{2}-1, & n\equiv2~(\bmod~4).
\end{aligned}
\right.
$$
\end{theorem}
 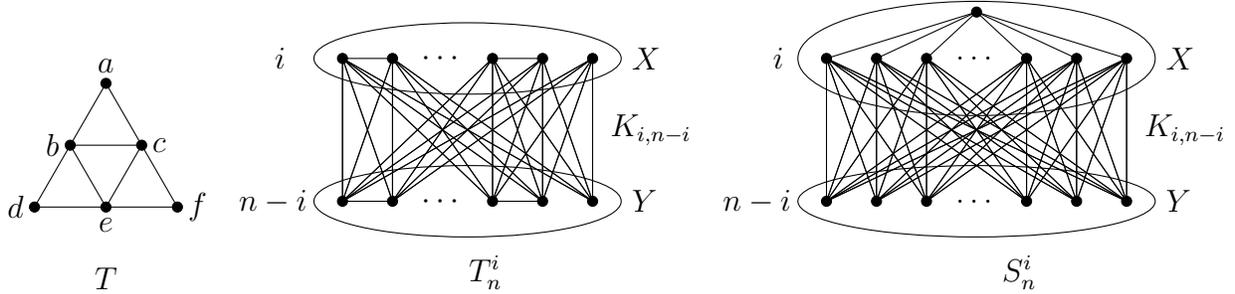
\begin{figure}[ht]
\begin{tikzpicture}[scale=.95]
\filldraw (0,{sqrt(3)}) node[above]{$a$} circle (2 pt) (1,0) node[right]{$f$} circle (2 pt) (-1,0) node[left]{$d$} circle (2 pt) (0,0) node[below]{$e$} circle (2 pt) (.5,{sqrt(3)/2}) node[right]{$c$} circle (2 pt) (-.5,{sqrt(3)/2}) node[left]{$b$} circle (2 pt) (0,-1) node{$T$} (0,2);

\draw (1,0) -- (-1,0) -- (0,{sqrt(3)}) -- (1,0) (0,0) -- (1/2,{sqrt(3)/2}) -- (-1/2,{sqrt(3)/2}) -- (0,0);

\end{tikzpicture}
\begin{tikzpicture}[scale=.95]

\foreach \x in {-.7,-1.4,.7,1.4,2.1}{
\foreach \y in {-.7,-1.4,.7,1.4,2.1}{

\filldraw (\y,0)  -- (\x,-2) ;

}
}

\draw (0,0) node{$\cdots$} (0,-2) node{$\cdots$};


\draw (-.7,0) -- (-1.4,0)  (.7,0) -- (1.4,0);
\draw (-.7,-2) -- (-1.4,-2) (.7,-2) -- (1.4,-2);


\draw  (2.5,-2) arc (0:360:2.15cm and .5cm) node[align=center,right]{$Y$};

\draw (-2.5,0) node[right]{$i$};

\draw (-3,-2) node[right]{$n-i$};

\draw (2.2,-1) node[right]{$K_{i,n-i}$};

\draw (2.5,0) arc (0:360:2.15cm and .5cm) node[align=center,right]{$X$};

\draw (0.6,-3) node{$T^i_n$};

\foreach \x in {-.7,-1.4,.7,1.4,2.1}{

\filldraw (\x,0) circle (2pt)  (\x,-2) circle (2pt);

}

\end{tikzpicture}
\begin{tikzpicture}[scale=.95]

\foreach \x in {-2.1,-1.4,-.7,.7,1.4,2.1}{
\foreach \y in {-2.1,-1.4,-.7,.7,1.4,2.1}{

\filldraw (\x,0)  -- (\y,-2) ;

}
}

\foreach \x in {-2.1,-1.4,-.7,.7,1.4,2.1}{

\filldraw  (\x,0) circle (2pt) (\x,-2) circle (2pt);
\draw (\x,0) -- (0,.65);
}

\draw (0,0) node {$\dots$} (0,-2) node{$\dots$};

\filldraw (0,.65) circle (2pt);



\draw  (2.5,-2) arc (0:360:2.5cm and .5cm) node[align=center,right]{$Y$};

\draw (2.2,-1) node[right]{$K_{i,n-i}$};

\draw (-3,0) node[right]{$i$};

\draw (-3.7,-2) node[right]{$n-i$};

\draw (2.5,0) arc (0:360:2.5cm and .8cm) node[align=center,right]{$X$};

\draw (0.6,-3) node{$S^i_n$};

\end{tikzpicture}
 \caption{Graphs $T$, $T^{i}_{n}$ and $S^{i}_{n}$}
 \label{T_n}
\end{figure}

\begin{definition}
Let $T^{i}_{n}$ denote a graph obtained from a complete bipartite graph $K_{i,n-i}$ plus a maximum matching in the class $X$ which has $i$ vertices and a maximum matching in the class $Y$ which has $n-i$ vertices, see $T^i_n$ in Figure \ref{T_n}. Let $S^{i}_{n}$ denote a graph obtained from $K_{i,n-i}$ plus an $i$-vertex star in the class $X$, see $S^{i}_{n}$ in Figure \ref{T_n}.
\end{definition}
\begin{theorem}\label{ET}
Let $n$ $(n\neq5,6)$ be a natural number,

when $n\equiv0~(\bmod~4)$, the extremal graph for $T$ is $T_n^{\frac{n}{2}}$,

 when $n\equiv1~(\bmod~4)$, the extremal graphs for $T$ are $T^{\lceil\frac{n}{2}\rceil}_{n}$ and $S^{\lceil\frac{n}{2}\rceil}_{n}$,
 
 when $n\equiv2~(\bmod~4)$, the extremal graphs for $T$ are $T^{\frac{n}{2}}_{n}$, $T^{\frac{n}{2}+1}_{n}$ and $S^{\frac{n}{2}}_{n}$,
 
  when $n\equiv3~(\bmod~4)$, the extremal graphs for $T$ are $T^{\lceil\frac{n}{2}\rceil}_{n}$ and $S^{\lceil\frac{n}{2}\rceil}_{n}$.
\end{theorem}
These two results are known for sufficiently large $n's$ \cite{LIU}, here we are able to determine the value for small $n's$.

Using Theorems \ref{T} and \ref{ET}, we are able to prove the next two results for $P^2_6$.
\begin{theorem}\label{P6}
The maximum number of edges in an $n$-vertex $P^{2}_{6}$-free graph $(n\neq5)$ is:
\begin{displaymath} \ex(n,P^{2}_{6})=\left\{
\begin{aligned}
&\left\lfloor\frac{n^{2}}{4}\right\rfloor+\left\lfloor\frac{n-1}{2}\right\rfloor,~n\equiv1,2,3 ~(\bmod~6), \\
&\left\lfloor\frac{n^{2}}{4}\right\rfloor+\left\lceil\frac{n}{2}\right\rceil,~\otherwise.
\end{aligned}
\right.
\end{displaymath}
\end{theorem}
\begin{figure}[b]
\centering
\begin{tikzpicture}

\foreach \x in {-2.1,-1.4,...,2.8}{
\foreach \y in {-2.1,-1.4}{

\filldraw (\x,0)  -- (\y,-2) ;

}
\foreach \y in {0.7,1.4,...,2.7}{
\filldraw (\x,0)  -- (\y,-2) ;

}
}

\foreach \x in {-2.1,-1.4,...,2.8}{

\filldraw (\x,0) circle (2pt);

}

\foreach \x in {-2.1,-1.4}{

\filldraw (\x,-2) circle (2pt);

}

\foreach \x in {0.7,1.4,...,2.7}{

\filldraw (\x,-2) circle (2pt);

}


\draw (-2.1,0) -- (-1.4,.4) --  (-1.4,0)  (-1.05,0)  (-1.4,.4) --  (-1.75,0)  (-1.4,.4) --  (-.7,0)     (.35,.4) -- (0,0)  -- (.7,0)  -- (.35,.4)  (1.75,.4)-- (1.4,0) -- (2.1,0) --  (1.75,.4) (1.05,.2) node{$\cdots$} ;

\begin{small}
\filldraw (-1.05,0) node{$\cdots$}  (-1.4,.4) circle (2pt) (-1.75,0) circle (2pt) (1.75,.4) circle (2pt) (.35,.4) circle (2pt);
\end{small}
\draw (-.35,-2) node{$\dots$} ;


\draw  (2.5,-2) arc (0:360:2.5cm and .5cm) node[align=center,right]{$Y$};

\draw (-3,0) node[right]{$i$};

\draw (-3.7,-2) node[right]{$n-i$};

\draw (2.5,-1) node[right]{$K_{i,n-i}$};

\draw (2.5,0) arc (0:360:2.5cm and .8cm) node[align=center,right]{$X$};

\draw (0,-3) node{$F^{i,j}_n$};

\end{tikzpicture}\qquad
\begin{tikzpicture}

\foreach \x in {-2.1,-1.4,...,2.1}{
\foreach \y in {-2.1,-1.4}{

\filldraw (\x,0)  -- (\y,-2) ;

}
\foreach \y in {0.7,1.4}{

\filldraw (\x,0)  -- (\y,-2) ;

}

}

\foreach \x in {-2.1,-1.4,...,2.1}{

\filldraw  (\x,0) circle (2pt);
}

\foreach \x in {-2.1,-1.4}{

\filldraw  (\x,-2) circle (2pt);
}

\foreach \x in {0.7,1.4}{

\filldraw  (\x,-2) circle (2pt);
}


\draw (-1.75,.4) -- (-2.1,0) -- (-1.4,0) -- (-1.75,.4)   (-.35,.4) -- (-.7,0) -- (0,0) -- (-.35,.4) (1.05,.4) -- (.7,0) -- (1.4,0) -- (1.05,.4)  (-1.05,.2) node{$\cdots$}  (-.35,-2) node{$\cdots$}  ;
\filldraw (-1.75,.4) circle (2pt)  (-.35,.4) circle (2pt) (1.05,.4) circle (2pt);

\draw  (1.8,-2) arc (0:360:2.15cm and .5cm) node[align=center,right]{$Y$};

\draw (-3.7,-2) node[right]{$n-i$};
\draw (-3,0) node[right]{$i$};
\draw (1.8,-1) node[right]{$K_{i,n-i}$};

\draw (1.8,0) arc (0:360:2.15cm and .8cm) node[align=center,right]{$X$};

\draw (-.35,-3) node{$H^{i}_{n}$};

\end{tikzpicture}
\caption{Graphs $F^{i,j}_{n}$ and $H^{i}_{n}$ }
\label{F^{i}_{n}}
\end{figure}
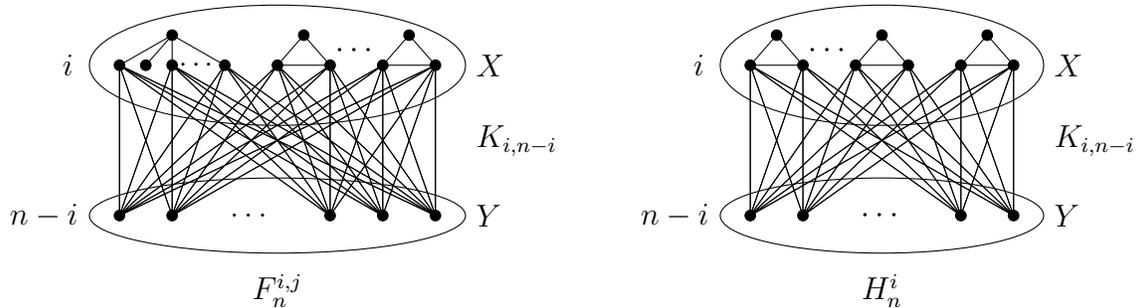
\begin{definition}
Suppose $3 \centernot\mid n$, and $1\leq j\leq i$. Let $F^{i,j}_{n}$ be the graph obtained by adding vertex disjoint triangles (possibly 0) and one star with $j$ vertices in the class $X$ of size $i$ of $K_{i,n-i}$, see Figure \ref{F^{i}_{n}} (Of course $3 \mid (i-j)$ is supposed). On the other hand if  $3 \mid i$ then add $\frac{i}{3}$ vertex disjoint triangles in the class $X$ of size $i$. The so obtained graph is denoted by $H^{i}_{n}$, see Figure \ref{F^{i}_{n}}.
\end{definition}
\begin{theorem}\label{EP6}
Let $n\geq 6$ be a natural number. The extremal graphs for $P_6^2$ are the following ones.

When $n\equiv 1 \pmod{6}$ then $F_n^{\lceil {\frac{n}{2}}\rceil, j}$ and $H_n^{\lfloor {\frac{n}{2}}\rfloor}$,

when $n\equiv 2 \pmod{6}$ then $F_n^{{\frac{n}{2}}, j}$ and $F_n^{{\frac{n}{2}}+1, j}$,

when $n\equiv 3 \pmod{6}$ then $F_n^{\lceil {\frac{n}{2}}\rceil, j}$ and $H_n^{\lceil {\frac{n}{2}}\rceil+1}$,

when $n\equiv 0, 4, 5 \pmod{6}$ then $H_n^{{\frac{n}{2}}}$, $H_n^{{\frac{n}{2}} +1}$ and $H_n^{\lceil {\frac{n}{2}}\rceil}$, respectively. $(j$ can have all the values satisfying the conditions $j\leq i$ and $3 \mid (i-j))$.
\end{theorem}

On the basis of these results let us pose a conjecture for the general case.
\begin{conjecture}
$${\rm ex}(n,P_k^2)\leq \max \left\{\frac{ i\left( \left\lfloor {2k\over 3}\right\rfloor -2)\right)}{2} +i(n-i)\right\}.$$

If $\left\lfloor {2k\over 3}\right\rfloor -1$ divides $i$ then the following graph gives equality here. Take a complete bipartite graph with parts of size $i$ and $n-i$, add vertex disjoint complete graphs on $\left\lfloor {2k\over 3}\right\rfloor -1$ vertices to the part with $i$ elements.
\end{conjecture}
Observe that Theorems 1, 2, 3 and 7 justify our conjecture for the cases when $k=3,4,5,6.$ We will give some hints in Section 3 how we arrived to this conjecture. A weaker form of this conjecture is the following one.
\begin{conjecture}
$${\rm ex}(n,P_k^2)= {n^2\over 4} + {\left( \left\lfloor {2k\over 3}\right\rfloor -2\right)n \over 2 }+o(1)$$
where $o(1)$ depends only on $k$.
\end{conjecture}
\section{Proofs of the main results}
\subsection{The Tur\'an number and the extremal graphs for $P^2_5$}
\begin{proof}[Proof of Theorem \ref{P5}]
The fact that $\ex(n, P^{2}_{5})\geq\left\lfloor \frac{n^2+n}{4}\right\rfloor$ follows from the construction $E^{\left\lceil\frac{n}{2}\right\rceil}_n$.\\
We prove the inequality
\begin{eqnarray}\label{P_5}
\aligned
\ex(n, P^{2}_{5})\leq \left\lfloor\frac{n^{2}+n}{4}\right\rfloor ~(n\geq5)
\endaligned
\end{eqnarray}
by induction on $n$.

We check the base cases first. Since our induction step will go from $n-4$ to $n$, we have to find a base case in each residue class $\bmod~4$.

Let $G$ be an $n$-vertex $P^{2}_{5}$-free graph.
When $n\leq3$, $K_n$ is the graph with the most number of edges and does not contain $P^{2}_5$, $e(K_n)\leq \left\lfloor\frac{n^{2}+n}{4}\right\rfloor$. This settles the cases $n=1,2,3$. However, when $n=4$, $e(K_{4})=6>\lfloor\frac{4^{2}+4}{4}\rfloor$, the statement is not true. Then we show that the statement is true for $n=8$. If $P^2_4 \nsubseteq G$, $e(G)\leq\lfloor\frac{8^{2}}{4}\rfloor$. If $P^2_4\subseteq G$ and $K_4\nsubseteq G$, each vertex $v\in V(G-P^2_4)$ can be adjacent to at most 2 vertices of the copy of $P^2_4$, since $e(G-P^2_4)\leq5$, we have $e(G)\leq5+8+5\leq18=\lfloor\frac{8^{2}+8}{4}\rfloor$. If $K_4\subseteq G$, then each vertex $v\in V(G-K_4)$ can be adjacent to at most one vertex of the $K_4$, since $e(G-P^2_4)\leq6$, we have $e(G)\leq16$.

 Suppose (\ref{P_5}) holds for all $k\leq n-1$, the proof is divided into 3 parts,

 \textbf{Case 1}.\ \ If $P^{2}_{4}\nsubseteq G$, then by Theorem \ref{DIR}, $e(G)\leq \lfloor \frac{n^{2}}{4}\rfloor$.

\textbf{Case 2}.\ \ If $P^{2}_{4}\subseteq G$ and $K_{4}\nsubseteq G$, then each vertex $v\in V(G-P^{2}_{4})$ can be adjacent to at most 2 vertices of the copy of $P^{2}_{4}$, otherwise, $P^{2}_{5}\subseteq G$. Since $G-P^{2}_{4}$ is an $(n-4)$-vertex $P^{2}_{5}$-free graph, we have\\
\centerline{$e(G)\leq 5+2(n-4)+e(G-P^{2}_{4})\leq2n-3+\ex(n-4,P^{2}_{5})$.}
By the induction hypothesis, $\ex(n-4, P^{2}_{5})\leq \left\lfloor\frac{(n-4)^{2}+n-4}{4}\right\rfloor$  then
\begin{eqnarray}
e(G)\leq 2n-3+\ex(n-4,P^{2}_{5})\leq 2n-3+\left\lfloor\frac{(n-4)^{2}+n-4}{4}\right\rfloor
 =\left\lfloor\frac{n^{2}+n}{4}\right\rfloor~(n\geq5).\end{eqnarray}
 
\textbf{Case 3}.\ \ If $K_{4}\subseteq G$, then each vertex $v\in V(G-K_{4})$ can be adjacent to at most one vertex of the $K_{4}$, otherwise, $P^{2}_{5}\subseteq G$. Since $G-K_{4}$ is an $(n-4)$-vertex $P^{2}_{5}$-free graph, we have\\
\centerline{$e(G)\leq 6+(n-4)+e(G-K_{4})\leq n+2+\ex(n-4,P^{2}_{5})$.}
By the induction hypothesis, $\ex(n-4, P^{2}_{5})\leq \left\lfloor\frac{(n-4)^{2}+n-4}{4}\right\rfloor$, thus
\begin{eqnarray}e(G)\leq n+2+\left\lfloor\frac{(n-4)^{2}+n-4}{4}\right\rfloor
 =5+\left\lfloor\frac{n^{2}-3n}{4}\right\rfloor\leq\left\lfloor\frac{n^{2}+n}{4}\right\rfloor~(n\geq5).\end{eqnarray}\qedhere\end{proof}
\begin{proof}[Proof of Theorem \ref{EP5}]
We determine the extremal graphs for $P^{2}_{5}$ by induction on $n$. Let $G$ be an $n$-vertex $P^2_5$-free graph satisfying (1) with equality. It is easy to check, when $n=5$, the
extremal graphs for $P^2_{5}$ are $G_0$, $E^2_{5}$ and $E^3_{5}$. When $n=6,7,8$,  the extremal graphs for $P^2_{5}$ are $E^3_{6}$ and $E^4_{6}$, $E^4_{7}$, $E^4_{8}$, respectively.

Suppose Theorem \ref{EP5} is true for $k\leq n-1$, when $n\geq9$, the proof is divided into 3 parts.

 \textbf{Case 1}.\ \ If $P^{2}_{4}\nsubseteq G$, the equality in (1) cannot hold, then we cannot find any extremal graph for $P^2_5$ in this case.
 
 \textbf{Case 2}.\ \ If $P^{2}_{4}\subseteq G$ and $K_{4}\nsubseteq G$,
 the equality holds in inequality (2) if and only if each vertex $v\in V(G-P^{2}_{4})$ is adjacent to 2 vertices of the $P^{2}_{4}$ and $G-P^{2}_{4}$ is an extremal graph on $n-4$ vertices for $P^2_5$.
Let $a,~b,~c$ and $d$ be four vertices of a copy of $P^{2}_{4}$, $d_{P^{2}_{4}}(b)=d_{P^{2}_{4}}(c)=3$. 
By the induction hypothesis, $G-P^{2}_{4}$ is obtained from a complete bipartite graph $K_{i,n-4-i}$ plus a maximum matching in $X^{'}$, where $X^{'}$ is the class of $G-P^{2}_{4}$ with size $i$. 
It is easy to check that every vertex $v\in V(G-P^{2}_{4})$ can be adjacent to either $a$ and $d$ or $b$ and $c$. 

Since $|V(G-P^{2}_{4})|\geq5$, we have $|V(X{'})|\geq2$. The endpoints of an edge in $G-P^2_4$ cannot be both adjacent to $b$ and $c$, otherwise, they form a $K_4$. Also, the endpoints of an edge in $G-P^2_4$ which have one end vertex as a matched vertex in $X{'}$ and one end vertex in $Y^{'}$ can be both adjacent to none of $\{a,b,c\}$ and $d$, otherwise, these would create a $P^2_5$. If there exists a matched vertex $v\in X{'}$ which is adjacent to $b$ and $c$, then all vertices $w\in N(v)$ should be adjacent to $a$ and $d$, these form a $P^2_5$. Hence, it is only possible that all matched vertices in $X^{'}$ are adjacent to both $a$ and $d$, all vertices in $Y^{'}$ are adjacent to $b$ and $c$. When there exists an unmatched vertex $v_{0}\in X^{'}$, since $N(v_{0})=Y^{'}$, if $v_{0}$ is adjacent to $b$ and $c$, we have $P^{2}_{5}\subseteq G$. Thus $G$ is obtained from a complete bipartite graph $K_{i+2,n-i-2}$ plus a maximum matching in $X$, where $X=X^{'}\cup\{b,c\}$ and $Y=Y^{'}\cup a\cup d$. Therefore, if $G-P^2_4$ is $E^{\lceil\frac{n-4}{2}\rceil}_{n-4}$ then $G$ is $E^{\lceil\frac{n}{2}\rceil}_n$, if $E^{\lfloor\frac{n-4}{2}\rfloor}_{n-4}$ then $G$ is $E^{\lfloor\frac{n}{2}\rfloor}_n$.

 \textbf{Case 3}.\ \ If $K_{4}\subseteq G$, the inequality in (3) can be equality only when $n=5$ and the vertex $v\in~V(G-K_{4})$ is adjacent to one vertex of the $K_{4}$, that is $G_0$.
 \end{proof}
 \subsection{The Tur\'an number and the extremal graphs for $T$}
To prove Theorem \ref{T}, we need the following lemmas.
\begin{lemma}\label{1}
Let $G$ be an $n$-vertex $T$-free nonempty graph such that for each edge $\{x,y\}\in E(G)$, $d(x)+d(y)\geq n+2$ holds, then we have $K_{4}\subseteq G$.
\end{lemma}
\begin{proof}
From the condition we know that each edge belongs to at least two triangles. Let $abc$ and $bcd$ be two triangles, if $a$ is adjacent to $d$ then $a,b,c$ and $d$ induce a $K_{4}$, if not, since edge $\{b,d\}$ is contained in at least two triangles, there exists at least one vertex $e$ such that $bde$ is a triangle. Similarly, edge $\{c,d\}$ is also contained in at least two triangles, then,
either there exists a vertex $f$ which is adjacent to $c$ and $d$, this implies that vertices $a,b,c,d,e$ and $f$ induce a $T$, or $c$ is adjacent to $e$, this implies that vertices $b,c,d$ and $e$ induce a $K_{4}$.
\end{proof}
\begin{lemma}\label{2}
Let $G$ be an $n$-vertex $(n\geq7)$ $T$-free graph and $K_{4}\subseteq G$, then $e(G)\leq2n-2+ex(n-4,T)$. For $n\geq8$, the equality might hold only if each vertex $v\in V(G-K_{4})$ is adjacent to 2 vertices of the $K_{4}$.
\end{lemma}
\begin{proof}
If there exists vertex $v\in V(G-K_4)$, such that $v$ is adjacent to at least 3 vertices of the $K_{4}$, it is simple to check that every other vertex $u\in V(G-K_4)$ can be adjacent to at most one vertex of the $K_{4}$, otherwise $T\subseteq G$, then $e(G)\leq6+4+(n-5)+e(G-K_{4})\leq n+5+ex(n-4,T)$. If not, each vertex in $G-K_{4}$ is adjacent to
at most 2 vertices of the $K_{4}$, then $e(G)\leq6+2(n-4)+e(G-K_{4})\leq 2n-2+ex(n-4,T)$. When $n\geq8$, $e(G)\leq 2n-2+ex(n-4,T)$, the equality holds only if each vertex $v\in V(G-K_{4})$ is adjacent to 2 vertices of the $K_{4}$.
\end{proof}
\begin{proof}[Proof of Theorem \ref{T}]
Let $$f_{T}(n)=\left\{
\begin{aligned}
&\left\lfloor\frac{n^{2}}{4}\right\rfloor+\left\lfloor\frac{n}{2}\right\rfloor,~n\not\equiv2~(\bmod~4), \\
&\frac{n^{2}}{4}+\frac{n}{2}-1,~n\equiv2~(\bmod~4).
\end{aligned}
\right.$$

The fact that $\ex(n,T)\geq f_{T}(n)$ follows from the construction $T^{\left\lceil\frac{n}{2}\right\rceil}_{n}$. Next, we show the inequality
\begin{eqnarray} \ex(n,T)\leq f_{T}(n)
\end{eqnarray}
by induction on $n$. 

Let $G$ be an $n$-vertex $T$-free graph.
first, we show the induction steps, in the end we will show the base cases which are needed to complete the induction.

Suppose $(4)$ holds for all $l\leq n-1$, in the following cases, we will assume that $k\geq2$, the proof is divided into 4 cases.

\textbf{Case 1}.\ \ When $n=4k$, we divide the proof of $\ex(4k,T)\leq f_{T}(4k)=4k^{2}+2k$ into 2 subcases. Let $G$ be a $4k$-vertex $T$-free graph.\\
$(i)$\ If $\delta (G)\leq2k+1$, after removing a vertex of minimum degree and by the induction hypothesis $\ex(4k-1,T)=4k^{2}-1$, we get \vspace{-0.5em}\begin{eqnarray}e(G)\leq \ex(4k-1,T)+2k+1\leq4k^{2}-1+2k+1=f_{T}(4k).\end{eqnarray}
$(ii)$\  If $\delta (G)\geq2k+2$, then for each edge $\{u,v\}\in E(G)$, $d(u)+d(v)\geq4k+4$. By Lemmas \ref{1} and \ref{2} and the induction hypothesis $\ex(4k-4,T)=4(k-1)^2+2(k-1)$, we get \vspace{-0.5em}\begin{eqnarray}e(G)\leq2n-2+\ex(4k-4,T)=8k-2+4(k-1)^2+2(k-1)=f_{T}(4k).\end{eqnarray}
Therefore, $\ex(4k,T)\leq f_{T}(4k)$.

\textbf{Case 2}.\ \ When $n=4k+1$, we divide the proof of $\ex(4k+1,T)\leq f_{T}(4k+1)=4k^{2}+4k$ into 3 subcases. Let $G$ be a $(4k + 1)$-vertex $T$-free graph.\\
\textbf{($i$)}\ If $\delta (G)\leq2k$, after removing a vertex of minimum degree and by the induction hypothesis $\ex(4k,T)=4k^{2}+2k$, we have \vspace{-0.5em}\begin{eqnarray}e(G)\leq \ex(4k,T)+2k\leq f_{T}(4k+1).\end{eqnarray}
Now, we assume that in the following two cases $\delta(G)\geq2k+1$. Then for any pair of vertices $\{u,v\}\in E(G)$, $d(u)+d(v)\geq4k+2$ holds.\\
 \textbf{($ii$)}\  Suppose that there exists an edge $\{u,v\}\in E(G)$, such that $d(u)+d(v)=4k+2$. This implies that $u$ and $v$ have at least one common neighbor. Deleting $\{u,v\}$ we can use the induction hypothesis $\ex(4k-1,T)=4k^2-1$. Then
 \vspace{-0.5em}\begin{eqnarray}e(G)\leq4k+1+\ex(4k-1,T)=f_{T}(4k+1).\end{eqnarray}
\textbf{($iii$)}\
For each edge $\{u,v\}\in E(G)$, $d(u)+d(v)\geq4k+3$ holds. By Lemmas \ref{1} and \ref{2} and the induction hypothesis $\ex(4k-3,T)=4(k-1)^2+4(k-1)$ we get
\vspace{-0.5em}\begin{eqnarray}e(G)\leq2n-2+\ex(4k-3,T)=8k+4(k-1)^2+4(k-1)=f_{T}(4k+1).\end{eqnarray}
Therefore, $\ex(4k+1,T)\leq f_{T}(4k+1)$.

\textbf{Case 3}.\ \  When $n=4k+2$, we divide the proof of $\ex(4k+2,T)\leq f_{T}(4k+2)=4k^{2}+6k+1$ into 2 subcases. Let $G$ be a $(4k + 2)$-vertex $T$-free graph.\\
$(i)$\ If $\delta (G)\leq2k+1$, after removing a vertex of minimum degree and by the induction hypothesis $\ex(4k+1,T)=4k^{2}+4k$, we get \vspace{-0.5em}\begin{eqnarray}e(G)\leq \ex(4k+1,T)+2k+1\leq4k^{2}+6k+1=f_{T}(4k+2).\end{eqnarray}
\textbf{$(ii)$}\  If $\delta (G)\geq2k+2$, then for each edge $\{u,v\}\in E(G)$, $d(u)+d(v)\geq4k+4$. By Lemmas \ref{1} and \ref{2} and the induction hypothesis $\ex(4k-2,T)=4(k-1)^2+6(k-1)+1$, we get
\vspace{-0.5em}\begin{eqnarray}e(G)\leq2n-2+\ex(4k-2,T)=8k+2+4(k-1)^2+6(k-1)+1=f_{T}(4k+2). \end{eqnarray}
Therefore, $\ex(4k+2,T)\leq f_{T}(4k+2)$.

\textbf{Case 4}.\ \ When $n=4k+3$, we divide the proof of $\ex(4k+3,T)\leq f_{T}(4k+3)=4k^{2}+8k+3$ into 2 subcases. Let $G$ be a $(4k + 3)$-vertex $T$-free graph.\\
$(i)$\ If $\delta (G)\leq2k+2$, after removing a vertex of minimum degree and by the induction hypothesis $\ex(4k+2,T)=4k^{2}+6k+1$, we get
\vspace{-0.5em}\begin{eqnarray}e(G)\leq \ex(4k+2,T)+2k+2\leq4k^{2}+8k+3=f_{T}(4k+3).\end{eqnarray}
$(ii)$.\
If $\delta (G)\geq2k+3$, then for each edge $\{u,v\}\in E(G)$, $d(u)+d(v)\geq4k+6$. By Lemmas \ref{1} and \ref{2} and the induction hypothesis $\ex(4k-1,T)=4(k-1)^2+8(k-1)+3$, we get
\vspace{-0.5em} \begin{eqnarray}e(G)\leq2n-2+\ex(4k-1,T)=8k+4+4(k-1)^2+8(k-1)+3=f_{T}(4k+3).\end{eqnarray}
Therefore, $\ex(4k+3,T)\leq f_{T}(4k+3)$.

Now we show the base cases which are needed to complete the induction steps. Since our induction steps will go from $n-1$ to $n$, $n-2$ to $n$ and $n-4$ to $n$, we will require to show the statement is true for cases when $n=3,4,6$ and 9. 

When $n\leq4$, $K_{n}$ is the graph with the most number of edges, and $e(K_{n})=f_{T}(n)$.

When $n=5$, $e(K_{5})=10>f_{T}(5)$, the statement is not true, but we will see that the statement is true for $n=9$.

When $n=6$, let $v$ be a vertex with minimum degree. If $\delta(G)=1$, since $e(G-v)\leq10$, we get $e(G)\leq11$. If $\delta(G)=2$ and $e(G)=12$, then the only possibility is that $G-v$ is $K_5$, but then $T\subseteq G$, and we have $e(G)\leq11$. Suppose now $\delta(G)\geq3$. If $K_4\subseteq G$ and there exists a vertex $u\in V(G-K_4)$ which is adjacent to at least 3 vertices of the copy of $K_4$, then $w\in V(G-K_4-u)$ can be adjacent to at most one vertex of the $K_4$, otherwise, $T\subseteq G$. This contradicts $\delta(G)\geq3$. Then in this case it is only possible that $\{u,w\}\in E(G)$ and both $u$ and $w$ are adjacent to 2 vertices of the $K_4$ which implies that $e(G)\leq11$. If $K_{4}\nsubseteq G$, then by Tur\'{a}n's Theorem, we have $e(G)\leq12$ and the Tur\'{a}n graph $T(6,3)$ is the unique $K_4$-free graph which has 12 edges, however, $T\subseteq T(6,3)$, then $e(G)\leq11=f_{T}(6)$. Summarizing: $e(G)\leq11\leq f_{T}(6)$.

When $n=9$, suppose first that there exists a pair of vertices $\{u,v\}\in E(G)$, such that $d(u)+d(v)\leq10$. Deleting $\{u,v\}$ and using $\ex(7,T)=15$, we get $e(G)\leq9+15=24=f_{T}(9)$. If for each pair of vertices $\{u,v\}\in E(G)$, $d(u)+d(v)\geq11$ holds, by Lemma \ref{1}, we obtain $K_4\subseteq G$. Let $G^{'}$ denote the graph $G-K_4$. If $e(G^{'})\leq8$, since the number of edges between $K_4$ and $G^{'}$ is at most 10, we have $e(G)\leq6+10+8=24$. If $e(G^{'})\geq9$, then $K_4\subseteq G^{'}$ and the vertex $w\in G^{'}-K_{4}$ is adjacent to at least 3 vertices of the copy of $K_4$ in $G^{'}$. This implies that each vertex from $G-G{'}$ can be adjacent to at most 1 vertex of $G^{'}-w$, then the number of edges between $G-G{'}$ and $G^{'}$ is at most 8, we can conclude that, $e(G)\leq6+8+10=24$, $e(G)\leq24=f_{T}(9)$.

It is easy to see that the case $n=7$ can be proved using $n=3$ and $n=6$ (Case 4). Similarly, the case $n=8$ follows by $n=7$ and $n=4$ (Case 1). Hence the cases $n=6,7,8,9$ are settled forming a good bases for the induction. 
\end{proof}
Now, we determine the extremal graphs for $T$.
\vspace{-0.5em}
\begin{proof}[Proof of Theorem \ref{ET}]
Similarly to the proof of Theorem \ref{T}, first, we show the induction steps, in the end we will show the base cases which are needed to complete the induction.

Suppose that the extremal graphs for $T$ are as shown in Theorem \ref{T} for $l\leq n-1$. In the following cases, we will assume that $k\geq2$. 

Let $G$ be an $n$-vertex $T$-free graph with $e(G) = f_T(n)$. The proof is divided into 4 cases following the steps of the proof of Theorem \ref{T}.\\
\textbf{Case 1}.\ \ When $n =4k$, $f_T(n) = 4k^2 + 2k$.

\emph{\textbf{(i)}} \  If $\delta(G)\leq2k+1$, the equality in (5) holds only when there exists a $v\in V(G)$, such that $d(v)=\delta(G)=2k+1$ and $G-v$ is an extremal graph for $T$ on $4k-1$ vertices. By the induction hypothesis, $G-v$ can be either $T^{2k}_{4k-1}$ or $S^{2k}_{4k-1}$. Let $X^{'}$ and $Y^{'}$ be the classes in $G-v$ with size $2k$ and $2k-1$, respectively.

When $G-v$ is $T^{2k}_{4k-1}$, it can be easily checked that $v$ cannot be adjacent to the two endpoints of an edge which have two matched vertices located in different classes, otherwise, $T\subseteq G$, see Figure \ref{case1,i}. Let $w$ be the unmatched vertex in $Y^{'}$. Since $d(v)=2k+1$, $N(v)$ must contain the unmatched vertex $w\in Y^{'}$, then the only way to avoid $T\subseteq G$ is choosing $N(v)=w\cup X^{'}$. Consequently, $G=T^{2k}_{4k}$ holds.\\
\begin{figure}[ht]
    \centering
\begin{tikzpicture}

\draw[thick,blue] (2.1,-2) -- (2.8,-2) -- (1.4,-2)  (2.8,-2) -- (1.4,0);

\draw  (-.7,0) node[above]{$x_{3}$} (.7,0) node[above]{$x_2$} (1.4,0) node[above]{$x_1$};

\draw (0.7,-2) node[below]{$y_2$} (1.4,-2) node[below]{$y_1$};
\foreach \x in {-.7,-1.4,.7,1.4}{

\filldraw (\x,0) circle (2pt)  (\x,-2) circle (2pt);

}
\filldraw   (2.8,-2) circle (2pt) node[below]{$v$};

\draw (0,0) node{$\cdots$} (0,-2) node{$\cdots$};


\draw (-.7,0) -- (-1.4,0)  (.7,0) -- (1.4,0);

\draw (-.7,-2) -- (-1.4,-2) (.7,-2) -- (1.4,-2);


\draw  (2.45,-2) arc (0:360:2.15cm and .65cm);

\draw (-2.15,0) node[left]{$X'$};

\draw (-1.6,-1) node[left]{$K_{2k,2k-1}$};
\draw (-2.15,-2) node[left]{$Y'$};
\draw (2.3,0) arc (0:360:2.15cm and .75cm);

\begin{scope}[xshift=6cm, yshift=-2cm]
\filldraw (0,{sqrt(3)}) node[above]{$x_{2}$} circle (2 pt) (1,0) node[below]{$v$} circle (2 pt) (-1,0) node[below]{$x_3$} circle (2 pt) (0,0) node[below]{$y_1$} circle (2 pt) (.5,{sqrt(3)/2}) node[right]{$x_1$} circle (2 pt) (-.5,{sqrt(3)/2}) node[left]{$y_2$} circle (2 pt)  (0,2);

\draw (1,0) -- (-1,0) -- (0,{sqrt(3)}) -- (1,0) (0,0) -- (1/2,{sqrt(3)/2}) -- (-1/2,{sqrt(3)/2}) -- (0,0);
\end{scope}

\end{tikzpicture}
    \caption{}
    \label{case1,i}
\end{figure}

When $G-v$ is $S^{2k}_{4k-1}$, let $x_1$ denote the center of the star in $X^{'}$. If $v$ is adjacent to the two endpoints of the edge $\{x_{i},y_{j}\}$ $(x_i\in X{'}, y_i\in Y{'}, 2\leq i\leq2k,~1\leq j\leq2k-1)$, then $T\subseteq G$ (see Figure \ref{v}). We obtained a contradiction. But $d(v)=2k+1$ implies that this is always the case. \\
\begin{figure}[ht]
    \centering
\begin{tikzpicture}

\draw[thick,blue] (1.4,-2) -- (2.8,-2)  (2.8,-2) -- (1.4,0);


\foreach \x in {-.7,.7,1.4}{
\draw (\x,0) -- (.35,.6);
}

\filldraw (.35,.6) circle (2pt);

\draw  (-.7,0) node[below]{$x_{2k}$} (.7,0) node[below]{$x_3$} (1.4,0) node[below]{$x_2$} (.35,.6) node[right]{$x_1$};

\foreach \x in {-.7,.7,1.4}{

\filldraw (\x,0) circle (2pt)  (\x,-2) circle (2pt);

}
\filldraw  (2.8,-2) circle (2pt) node[below]{$v$};

\draw (0.7,-2) node[below]{$y_2$} (1.4,-2) node[below]{$y_1$};

\draw (0,0) node{$\cdots$} (0,-2) node{$\cdots$};





\draw  (2.15,-2) arc (0:360:1.8cm and .5cm);

\draw (-1.8,0) node[left]{$X'$};

\draw (-1.5,-1) node[left]{$K_{2k,2k-1}$};

\draw (-1.8,-2) node[left]{$Y'$};
\draw (2.15,0) arc (0:360:1.8cm and .75cm);

\begin{scope}[xshift=6cm, yshift=-2cm]
\filldraw (0,{sqrt(3)}) node[above]{$v$} circle (2 pt) (1,0) node[below]{$x_3$} circle (2 pt) (-1,0) node[below]{$y_{2}$} circle (2 pt) (0,0) node[below]{$x_{1}$} circle (2 pt) (.5,{sqrt(3)/2}) node[right]{$y_1$} circle (2 pt) (-.5,{sqrt(3)/2}) node[left]{$x_2$} circle (2 pt) (0,2);

\draw (1,0) -- (-1,0) -- (0,{sqrt(3)}) -- (1,0) (0,0) -- (1/2,{sqrt(3)/2}) -- (-1/2,{sqrt(3)/2}) -- (0,0);
\end{scope}
\end{tikzpicture}
\caption{}
    \label{v}
\end{figure}

\emph{\textbf{(ii)}} \  If $\delta(G)\geq2k+2$, this implies that $e(G)\geq2k(2k+2)=4k^2+4k$, which contradicts the fact that $\ex(4k,T)=4k^2+2k$.

That is, $G$ can only be $T^{\frac{n}{2}}_{n}$.\\
\textbf{Case 2}.\ \ When $n =4k+1$, $f_T(n) = 4k^2 + 4k$.

\emph{\textbf{(i)}} \  If $\delta(G)\leq2k$, the equality in (7) holds only if there exists $v\in V(G)$, such that $d(v)=\delta(G)=2k$ and $G-v$ is an extremal graph for $T$ on $4k$ vertices. By the induction hypothesis, $G-v$ is $T^{2k}_{4k}$. All neighbors of $v$ should be located in the same class, otherwise, $T\subseteq G$, we get that $G$ is $T^{2k+1}_{4k+1}$, that is $T^{\lceil\frac{n}{2}\rceil}_{n}$.

If $\delta(G)\geq2k+1$, then for any pair of vertices $\{u,v\}\in V(G)$, $d(u)+d(v)\geq4k+2$. Here we distinguishing two subcases.

\emph{\textbf{(ii)}}\  Suppose that there exists an edge $\{u,v\}\in E(G)$ such that $d(u)+d(v)=4k+2$. The equality in (8) holds only if when $d(u)=d(v)=2k+1$ and $G-u-v$ is an extremal graph for $T$ on $4k-1$ vertices. By the induction hypothesis, $G-u-v$ can be either $T^{2k}_{4k-1}$ or $S^{2k}_{4k-1}$. Let $X^{'}$ and $Y^{'}$ be the classes in $G-u-v$ with size $2k$ and $2k-1$, respectively. 

 When $G-u-v$ is $T^{2k}_{4k-1}$, as in the previous case, neither $u$ nor $v$ can be adjacent to the two endpoints of an edge which have two matched vertices located in different classes, see Figure \ref{case1,i}. If $N(u)-v\neq X^{'}$, then $u$ is adjacent to the unmatched vertex $w$ in $Y^{'}$ and the other $2k-1$ neighbors of $u$ are all located in $X^{'}$, say, $N(u)-v-w=\{x_{1},\ldots x_{2k-1}\}$ and $\{x_{2k-1},x_{2k}\}\in E(X{'})$, otherwise, $T\subseteq G$. Since $|X^{'}|\geq4$, in this case, $v$ cannot be adjacent to $x_{i}$ $(1\leq i\leq 2k-2)$, otherwise, $T\subseteq G$, see Figure \ref{case2,ii}. Now $v$ should choose $2k$ neighbors among the rest $2k+1$ vertices in $V(G-u-v-\bigcup\limits_{i=1}^{2k-2}x_{i})$, which implies that $v$ is adjacent to the two endpoints of an edge which have two matched vertices locate in different classes as endpoints, then $T\subseteq G$. Hence, $N(u)-v=X^{'}$, similarly, $N(v)-u=X^{'}$. Thus, $G$ is $T^{2k+1}_{4k+1}=T^{2k}_{4k+1}$, that is $T^{\lceil\frac{n}{2}\rceil}_{n}$.
\begin{figure}[ht]
    \centering
\begin{tikzpicture}

\draw[thick,blue] (2.1,-2) -- (2.8,-2) -- (-.7,0) (2.8,-2) -- (.7,0) (2.8,-2) -- (1.4,0);

\draw[dashed,blue] (1.4,0) -- (3.5,-1);

\draw (-1.4,0) node[below]{$x_{2k}$} (-.7,0) node[above]{$x_{2k-1}$} (.7,0) node[below]{$x_2$} (1.4,0) node[below]{$x_1$};

\foreach \x in {-.7,-1.4,.7,1.4}{

\filldraw (\x,0) circle (2pt)  (\x,-2) circle (2pt);

}
\filldraw (2.1,-2) circle (2pt) node[below]{$w$}  (2.8,-2) circle (2pt) node[below]{$u$} -- (3.5,-1) circle (2pt) node[below]{$v$};

\draw (0,0) node{$\cdots$} (0,-2) node{$\cdots$};


\draw (-.7,0) -- (-1.4,0)  (.7,0) -- (1.4,0);

\draw (-.7,-2) -- (-1.4,-2) (.7,-2) -- (1.4,-2);


\draw  (2.45,-2) arc (0:360:2.15cm and .65cm);

\draw (-2.15,0) node[left]{$X'$};

\draw (-1.6,-1) node[left]{$K_{2k,2k-1}$};
\draw (-2.15,-2) node[left]{$Y'$};
\draw (2.3,0) arc (0:360:2.15cm and .75cm);

\begin{scope}[xshift=6cm, yshift=-2cm]
\filldraw (0,{sqrt(3)}) node[above]{$x_{2k-1}$} circle (2 pt) (1,0) node[below]{$v$} circle (2 pt) (-1,0) node[below]{$x_2$} circle (2 pt) (0,0) node[below]{$x_1$} circle (2 pt) (.5,{sqrt(3)/2}) node[right]{$u$} circle (2 pt) (-.5,{sqrt(3)/2}) node[left]{$w$} circle (2 pt)  (0,2);

\draw (1,0) -- (-1,0) -- (0,{sqrt(3)}) -- (1,0) (0,0) -- (1/2,{sqrt(3)/2}) -- (-1/2,{sqrt(3)/2}) -- (0,0);
\end{scope}

\end{tikzpicture}
    \caption{}
    \label{case2,ii}
\end{figure}

Let us now consider the case when $G-u-v$ is $S^{2k}_{4k-1}$. Let $x_{1}$ denote the center of the star in $X^{'}$. If $u$ is adjacent to the two endpoints of the edge $\{x_{i},y_{j}\}$ $(2\leq i\leq2k,~1\leq j\leq2k-1)$, then $T\subseteq G$. Thus, there are only two possibilities for $T\nsubseteq G$: $N(u)-v=X^{'}$ or $N(u)-v=Y^{'}\cup x_{1}$. The same holds for $v$ and it is easy to check that if $N(u)-v=N(v)-u$, then $T\subseteq G$. From the above, the only possibility for $T\nsubseteq G$ is that when $N(u)-v=X^{'}$ and $N(v)-u=Y^{'}\cup x_{1}$ or in the another way around, which implies that $G$ is $S^{2k+1}_{4k+1}$, that is $S^{\lceil\frac{n}{2}\rceil}_{n}$.

\emph{\textbf{(iii)}} \ Suppose that for each edge $\{u,v\}\in E(G)$, $d(u)+d(v)\geq4k+3$ holds. Let $d(v)=\delta(G)$, then either $d(v)=2k+1$ or $d(v)\geq2k+2$, but in both cases, each neighbor of $v$ has degree at least $2k+2$. Then all $4k+1$ vertices have degree at least $2k+1$, but $2k+1$ of them, which are the neighbors of $v$, have degree at least one larger. This implies that $e(G)\geq\frac{(4k+1)(2k+1)+2k+1}{2}=4k^2+4k+1$, which contradicts the fact that $\ex(4k+1,T)=4k^2+4k$.

 That is, $G$ can be either $T^{\lceil\frac{n}{2}\rceil}_{n}$ or $S^{\lceil\frac{n}{2}\rceil}_{n}$.\\
\textbf{Case 3}.\ \ When $n =4k+2$ we have $f_T(n) = 4k^2 + 6k+1$.

\emph{\textbf{(i)}}\  If $\delta(G)\leq 2k+1$, the equality holds in (10) only if there exists $v\in V(G)$, such that $d(v)=\delta(G)=2k+1$ and $G-v$ is an extremal graph for $T$ on $4k+1$ vertices. By the induction hypothesis, $G-v$ can be either $T^{2k+1}_{4k+1}$ or $S^{2k+1}_{4k+1}$. 

Suppose first that $G-v$ is $T^{2k+1}_{4k+1}$. Let $X^{'}$ any $Y^{'}$ be the classes in $G-v$ with size $2k+1$ and $2k$, $w$ be the unmatched vertex in $X^{'}$. The vertex $v$ cannot be adjacent to the two endpoints of an edge which have two matched vertices located in different classes. Since $d(v)=2k+1$,
there are two possibilities to avoid $T$: $N(v)=X^{'}$ or
$N(v)=Y^{'}\cup w$, which implies that $G$ is either $T^{2k+1}_{4k+2}$ or $T^{2k+2}_{4k+2}$, that is $T^{\frac{n}{2}}_{n}$ or $T^{\frac{n}{2}+1}_{n}$. 

When $G-v$ is $S^{2k+1}_{4k+1}$. Let $X^{'}$ be the class in $G-v$ which contains a star and $Y^{'}$ be the other class of the $G-v$. Also, let $x_1$ denote the center of the star in $X^{'}$. Since, $d(v)=2k+1$ and $v$ cannot be adjacent to the two endpoints of an edge which is not incident with $x_1$, we get either $N(v)=Y^{'}\cup x_1$ or $N(v)=X^{'}$. If $N(v)=X^{'}$, $G$ is $S^{2k+1}_{4k+2}$, that is $S^{\frac{n}{2}}_{n}$. If $N(v)=Y^{'}\cup x_1$, $G$ is $S^{2k+2}_{4k+2}$, that is $S^{\frac{n}{2}+1}_{n}$. It is easy to see that $S^{\frac{n}{2}+1}_{n}$ is isomorphic to $S^{\frac{n}{2}}_{n}$.

\emph{\textbf{(ii)}}\  If $\delta(G)\geq 2k+2$, then $e(G)\geq(k+1)(4k+2)=4k^2+6k+2$, which contradicts the fact that $\ex(4k+2,T)=4k^2+6k+1$.

Therefore, $G$ can be $T^{\frac{n}{2}}_{n}$, $T^{\frac{n}{2}+1}_{n}$ or $S^{\frac{n}{2}}_{n}$.\\
\textbf{Case 4}.\ \ When $n =4k+3$ we have $f_T(n) = 4k^2 + 8k+3$.

\emph{\textbf{(i)}}\ If $\delta(G)\leq 2k+2$, the equality holds in (12) only if there exists $v\in V(G)$, such that $d(v)=\delta(G)=2k+2$ and $G-v$ is an extremal graph for $T$ on $4k+2$ vertices. By the induction hypothesis, $G-v$ can be $T^{2k+1}_{4k+2}$, $T^{2k+2}_{4k+2}$ or $S^{2k+1}_{4k+2}$.\\
 When $G-v$ is $T^{2k+1}_{4k+2}$ or $T^{2k+2}_{4k+2}$, similarly to Case 1 $(i)$, $G$ can only be $T^{2k+2}_{4k+3}$, that is $T^{\left\lceil\frac{n}{2}\right\rceil}_{n}$. \\
 When $G-v$ is $S^{2k+1}_{4k+2}$, similarly to Case 2 $(ii)$, $G$ can only be $S^{2k+2}_{4k+3}$, that is $S^{\left\lceil\frac{n}{2}\right\rceil}_{n}$.
 
\emph{\textbf{(ii)}}\ If $\delta(G)\geq 2k+3$, then $e(G)\geq\frac{(2k+3)(4k+3)}{2}>4k^2+9k+4>4k^2+8k+3$, which contradicts the fact that $\ex(4k+3,T)=4k^2+8k+3$.

Therefore, in this case, $G$ is either $T^{\left\lceil\frac{n}{2}\right\rceil}_{n}$ or $S^{\left\lceil\frac{n}{2}\right\rceil}_{n}$.

 Now we check the base cases which are needed to complete the induction.
 
When $n=4$, $\ex(4,T)=6$, $K_4$ is the extremal graph which has the maximum number of edges on $4$ vertices that does not contain $T$ as a subgraph.

Although the Theorem does not hold for $n=6$, we determine the extremal graphs in this case because it will help us to determine them for some other $n$'s.

When $n=6$, $\ex(6,T)=11$. It follows from the proof of Theorem \ref{T}, when $\delta(G)=1$, the only extremal graph for $T$ is as shown in Figure \ref{T_6}$(a)$. When $\delta(G)=2$, the only extremal graph for $T$ is as shown in Figure \ref{T_6}$(b)$. Since $\delta(G)\geq4$ implies $e(G)\geq12$, this is not possible. The only remaining case is $\delta(G)=3$. When $\delta(G)=3$ and $K_4\subseteq G$, by case analysis we obtain that the extremal graphs for $T$ can be Figure \ref{T_6}$(c)$ and Figure \ref{T_6}$(d)$, which are $T^3_6$ and $T^4_6$. Suppose now that $\delta(G)=3$ and $K_4\nsubseteq G$. Let $d(v)=\delta(G)=3$, then $e(G-v)=8$, the only possibility is  that $G-v$ is $T(5,3)$. It is easy to check that $G$ can only be $S^3_6$, see Figure \ref{T_6}$(e)$.\\
\begin{figure}[ht]
    \centering
\begin{tikzpicture}[scale=.7] %
\foreach \xi in {0,72,...,288}{
\pgfmathsetmacro{\z}{\xi+72}
\foreach \yi in {\xi,\z,...,288}{

\filldraw  (xyz polar cs:angle=\xi,radius=1) circle (2pt) -- (xyz polar cs:angle=\yi,radius=1) circle (3pt);

}
}

\filldraw (1,0) -- (2,0) circle (3pt);
\draw (0,-1.5) node[below]{(a)};

\end{tikzpicture}\quad
\begin{tikzpicture}[scale=.7]
\foreach \xi in {36,108,...,287}{
\pgfmathsetmacro{\z}{\xi+72}
\foreach \yi in {\xi,\z,...,287}{

\filldraw  (xyz polar cs:angle=\xi,radius=1) circle (2pt) -- (xyz polar cs:angle=\yi,radius=1) circle (3pt);

}
}

\draw (xyz polar cs:angle=324,radius=1) -- (xyz polar cs:angle=108,radius=1) (xyz polar cs:angle=324,radius=1) -- (xyz polar cs:angle=180,radius=1);

\filldraw (xyz polar cs:angle=324,radius=1) -- (2,0) circle (3pt) -- (xyz polar cs:angle=36,radius=1);
\draw (0,-1.5) node[below]{(b)};
\end{tikzpicture}\quad
\begin{tikzpicture}
\draw (0,0) -- (0,1) (0,0) -- (1,1)  (1,0) -- (0,1) (1,0) -- (1,1) (0,0) -- (1,0)  (0,1) -- (1,1);
\draw [thick,blue] (2,0) -- (0,1) (2,0) -- (1,1) (2,0) -- (2,1) (0,0) -- (2,1) (1,0) -- (2,1);
\foreach \x in {0,1,2} {
\filldraw (\x,0) circle (2pt) (\x,1) circle (2pt);
}
\draw (1,-0.5) node[below]{(c)};
\end{tikzpicture}\quad
\begin{tikzpicture}
\draw (0,0) -- (0,1) (0,0) -- (1,0) (0,0) -- (1,1) (1,0) -- (0,1) (1,0) -- (1,1) (0,1) -- (1,1);
\draw [thick,blue] (0,0) -- (2,1) (0,0) -- (3,1) (1,0) -- (3,1) (1,0) -- (2,1) (2,1) -- (3,1) ;
\foreach \x in {0,1} {
\filldraw (\x,0) circle (2pt) (\x,1) circle (2pt);
}
\filldraw (3,1) circle (2pt);
\filldraw (2,1) circle (2pt);
\draw (1,-0.5) node[below]{(d)};
\end{tikzpicture}\quad
\begin{tikzpicture}
\draw (0,0) -- (0,1) (0,0) -- (1,1.5) (0,0) -- (2,1) (1,0) -- (0,1) (1,0) -- (1,1.5) (1,0) -- (2,1) (0,1) -- (1,1.5) -- (2,1);
\draw [thick,blue] (2,0) -- (0,1) (2,0) -- (1,1.5) (2,0) -- (2,1);
\foreach \x in {0,2} {
\filldraw (\x,0) circle (2pt) (\x,1) circle (2pt);
}
\filldraw (1,0) circle (2pt);
\filldraw (1,1.5) circle (2pt);
\draw (1,-0.5) node[below]{(e)};
\end{tikzpicture}
\caption{Extremal graphs for $T$ when $n=6$.}
\label{T_6}
\end{figure}
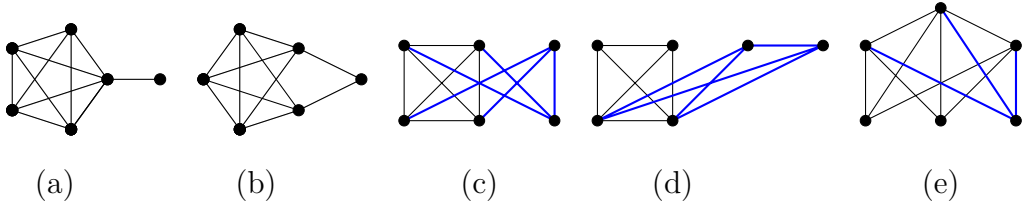

\begin{figure}[ht]
    \centering
\begin{tikzpicture}
\draw (0,0) -- (0,1) (0,0) -- (1,1)  (1,0) -- (0,1) (1,0) -- (1,1) (0,0) -- (1,0)  (0,1) -- (1,1);
\draw  (2,0) -- (0,1) (2,0) -- (1,1) (2,0) -- (2,1) (0,0) -- (2,1) (1,0) -- (2,1);
\foreach \x in {0,1,2} {
\filldraw (\x,0) circle (2pt) (\x,1) circle (2pt);
}
\draw (1,-0.5) node[below]{$T^3_6$};
\end{tikzpicture}\quad \quad
\begin{tikzpicture}
\draw (0,0) -- (0,1) (0,0) -- (1,1) (0,0) -- (2,1) (1,0) -- (0,1) (1,0) -- (1,1) (1,0) -- (2,1) (2,0) -- (0,1) (2,0) -- (1,1) (2,0) -- (2,1) (0,0) -- (1,0)  (0,1) -- (1,1);
\draw [thick,blue] (3,1) -- (0,0) (3,1) -- (1,0) (3,1) -- (2,0)  (2,1) --(3,1);
\foreach \x in {0,1,2} {
\filldraw (\x,0) circle (2pt) (\x,1) circle (2pt);
}
\filldraw  (3,1) circle (2pt);
\draw (1,-0.5) node[below]{(a)};
\end{tikzpicture}\quad\quad\quad\quad\quad
\begin{tikzpicture}
\draw (0,0) -- (0,1) (0,0) -- (1,0) (0,0) -- (1,1) (1,0) -- (0,1) (1,0) -- (1,1) (0,1) -- (1,1);
\draw (0,0) -- (2,1) (0,0) -- (3,1) (1,0) -- (3,1) (1,0) -- (2,1) (2,1) -- (3,1) ;
\foreach \x in {0,1} {
\filldraw (\x,0) circle (2pt) (\x,1) circle (2pt);
}
\filldraw (3,1) circle (2pt);
\filldraw (2,1) circle (2pt);
\draw (1,-0.5) node[below]{$T^4_6$};
\end{tikzpicture}\quad\quad
\begin{tikzpicture}
\draw (0,0) -- (0,1) (0,0) -- (1,1) (0,0) -- (2,1) (1,0) -- (0,1) (1,0) -- (1,1) (1,0) -- (2,1) (0,0) -- (1,0)  (0,1) -- (1,1) (3,1) -- (0,0) (3,1) -- (1,0) (2,1) --(3,1);
\draw [thick,blue] (3,1) -- (2,0) (2,0) -- (0,1) (2,0) -- (1,1)(2,0) -- (2,1);
\foreach \x in {0,1,2} {
\filldraw (\x,0) circle (2pt) (\x,1) circle (2pt);
}
\filldraw  (3,1) circle (2pt);
\draw (1,-0.5) node[below]{(a)};
\end{tikzpicture}\quad\quad
\begin{tikzpicture}
\draw (0,0) -- (0,1) (0,0) -- (1,1.5) (0,0) -- (2,1) (1,0) -- (0,1) (1,0) -- (1,1.5) (1,0) -- (2,1) (0,1) -- (1,1.5) -- (2,1);
\draw (2,0) -- (0,1) (2,0) -- (1,1.5) (2,0) -- (2,1);
\foreach \x in {0,2} {
\filldraw (\x,0) circle (2pt) (\x,1) circle (2pt);
}
\filldraw (1,0) circle (2pt);
\filldraw (1,1.5) circle (2pt);
\draw (1,-0.5) node[below]{$S^3_6$};
\end{tikzpicture}\quad\quad
\begin{tikzpicture}
\draw (0,0) -- (0,1) (0,0) -- (1,1.5) (0,0) -- (2,1) (1,0) -- (0,1) (1,0) -- (1,1.5) (1,0) -- (2,1) (0,1) -- (1,1.5) -- (2,1);
\draw (2,0) -- (0,1) (2,0) -- (1,1.5) (2,0) -- (2,1);
\draw [thick,blue] (3,1.3) -- (0,0) (3,1.3) -- (1,0) (3,1.3) -- (2,0) (3,1.3) -- (1,1.5);
\foreach \x in {0,2} {
\filldraw (\x,0) circle (2pt) (\x,1) circle (2pt);
}
\filldraw (1,0) circle (2pt);
\filldraw (1,1.5) circle (2pt);
\filldraw (3,1.3) circle (2pt);
\draw (1,-0.5) node[below]{(b)};
\end{tikzpicture}
\caption{Extremal graphs for $T$ when $n=7$.}
\label{T_7}
\end{figure}
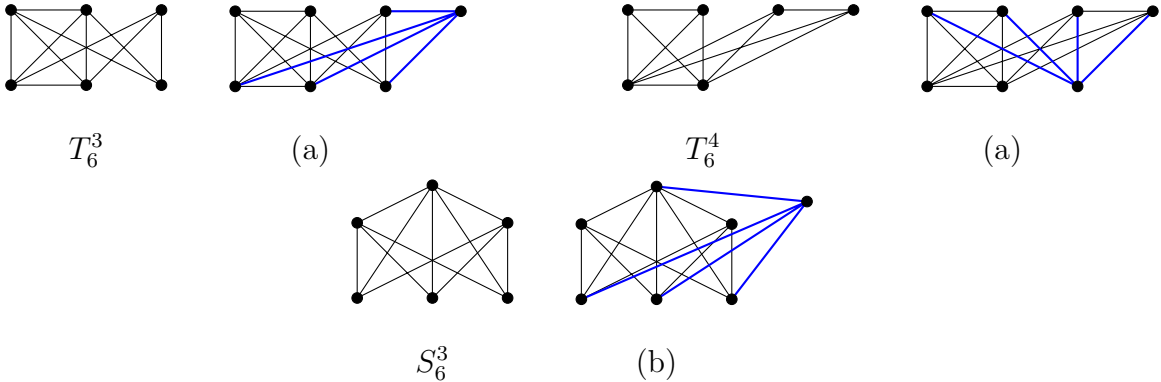

Suppose now that $n=7$, $\ex(7,T)=15$. It is not possible that $\delta(G)\leq3$, otherwise, $e(G)\leq 3+ex(6,T)=14$. Also, it is not possible that $\delta(G)\geq5$, otherwise, $e(G)>17$. Both are contradict with $e(G)=15$. Let $d(v)=\delta(G)$, the only possibility is that $\delta(G)=4$ and $G-v$ is a 6-vertex $T$-free graph. Since $d(v)=4$, we have $\delta(G-v)\geq3$, which implies that structures $(a)$ and $(b)$ in Figure \ref{T_6} are not possible. If $G-v$ is $T^3_6$ or $T^4_6$, then $G$ can only be $(a)$ in Figure \ref{T_7}, that is $T^4_7$. If $G-v$ is $S^3_6$, then $G$ can only be (b) in Figure \ref{T_7}, that is $S^4_7$.

Because case $n=8$ needs only the case $n=7$ (Case 1), case $n=9$ needs cases $n=7$ and $n=8$ (Case 2).
These base cases complete the proof.
\end{proof}
We will need the following statement later. It express that the "second best" graphs can be also well described if $4|n$.
\begin{proposition}\label{12}
 Let $n$ $(n\geq8)$ be a natural number such that $4|n$ and $G$ be an $n$-vertex $T$-free graph with $\frac{n^2}{4} + \frac{n}{2} - 1$ edges, then $G$ can only be $T^{\frac{n}{2}}_n$ minus an edge, $S^{\frac{n}{2}}_n$ or $S^{\frac{n}{2}+1}_n$.
\end{proposition}
\begin{proof}
We can suppose that $\delta(G) \leq \frac{n}{2}$, otherwise, $e(G) \geq \frac{n^2}{4} + \frac{n}{2}$. Let $v\in V(G)$ and $d(v)=\delta(G)$, then $e(G) \leq d(v) + \ex(n-1,T) \leq  \frac{n^2}{4}+\frac{n}{2}- 1,$ the equality holds only if $d(v) = \frac{n}{2}$ and $G-v$ is either $T^{\left\lceil\frac{n-1}{2}\right\rceil}_{n-1}$ or $S^{\left\lceil\frac{n-1}{2}\right\rceil}_{n-1}$. When $G-v$ is $T^{\left\lceil\frac{n-1}{2}\right\rceil}_{n-1}$, let $w$ be the unmatched vertex in $Y^{'}$ and $X^{'}=\{x_1,\dots,x_{\left\lceil\frac{n-1}{2}\right\rceil}\}$, $X^{'}$ and $Y^{'}$ be the classes of $G-v$ with size $\left\lceil\frac{n-1}{2}\right\rceil$ and $\left\lfloor\frac{n-1}{2}\right\rfloor$, respectively. Since $d(v)=\frac{n}{2}$ and $v$ cannot be adjacent to the two endpoints of an edge which have two matched vertices located in different classes, no matter $N(v)=X^{'}$ or $N(v)=X^{'}-x_{i}\cup w$ $(1\leq i\leq \left\lceil\frac{n-1}{2}\right\rceil)$,
$G$ is $T^{\frac{n}{2}}_n$ minus an edge in both cases. When $G-v$ is $S^{\left\lceil\frac{n-1}{2}\right\rceil}_{n-1}$, let $x_1$ be the center of the star in $X{'}$, $X^{'}=\{x_1,\dots,x_{\left\lceil\frac{n-1}{2}\right\rceil}\}$ and $Y^{'}=\{y_1,\dots,y_{\left\lfloor\frac{n-1}{2}\right\rfloor}\}$ be the classes of $G-v$ with size $\left\lceil\frac{n-1}{2}\right\rceil$ and $\left\lfloor\frac{n-1}{2}\right\rfloor$, respectively. Since $v$ cannot be adjacent to the two endpoints of the edge $\{ x_i,y_i\}$ $(2\leq i\leq \left\lceil\frac{n-1}{2}\right\rceil, 1\leq j\leq \left\lfloor\frac{n-1}{2}\right\rfloor)$ and $d(v)=\frac{n}{2}$, which implies that $N(v)=x_1\cup Y^{'}$ or $N(v)=X^{'}$. Therefore, $G$ can be either $S^{\frac{n}{2}}_n$ or $S^{\frac{n}{2}+1}_n$.
\end{proof}
\subsection{The Tur\'an number and the extremal graphs for $P^2_6$}
\begin{proof}[Proof of Theorem \ref{P6}]
Let
 $$
f_{P^{2}_{6}}(n)=\left\{
\begin{aligned}
&\left\lfloor\frac{n^{2}}{4}\right\rfloor+\left\lfloor\frac{n-1}{2}\right\rfloor,~n\equiv 1,2,3~(\bmod~6), \\
&\left\lfloor\frac{n^{2}}{4}\right\rfloor+\left\lceil\frac{n}{2}\right\rceil,~\text{otherwise}.
\end{aligned}
\right.
$$
The fact that $\ex(n, P^{2}_{6})\geq\left\lfloor\frac{n^{2}}{4}\right\rfloor+\left\lceil\frac{n}{2}\right\rceil$, when $n\equiv0,4,5~(\bmod~6)$, follows from the constructions $H^{\frac{n}{2}}_{n}$, $H^{\frac{n}{2}+1}_{n}$ and $H^{\lceil\frac{n}{2}\rceil}_{n}$, respectively.
 The fact that $\ex(n, P^{2}_{6})\geq\left\lfloor\frac{n^{2}}{4}\right\rfloor+\left\lfloor\frac{n-1}{2}\right\rfloor$, when $n\equiv1,2,3~(\bmod~6)$, follows from the constructions $F^{\lceil\frac{n}{2}\rceil,j}_{n}$.\\
It remains to show the inequality
\begin{eqnarray}\label{3}
\ex(n,P^{2}_{6})\leq f_{P^{2}_{6}}(n)
\end{eqnarray}
by induction on $n$.

Let $G$ be an $n$-vertex $P^{2}_{6}$-free graph. Since our induction step will go from $n-6$ to $n$, we have to find a base case in each residue class $\bmod~6$.

When $n\leq4$, $K_n$ is the graph with the most number of edges and $e(K_{n})=f_{P^{2}_{6}}(n)$. 

When $n=6$, if $P^{2}_{5}\nsubseteq G$, by Theorem \ref{P5}, $e(G)\leq\left\lfloor\frac{5^2+5}{4}\right\rfloor=7<f_{P^{2}_{6}}(6)$. 
If $P^{2}_{5}\subseteq G$, $K_{5}\nsubseteq G$ and $e(G)\geq13$, it can be checked that the vertex $v\in V(G-P^{2}_{5})$ can be adjacent to at most 3 vertices of the copy of $P^{2}_{5}$, otherwise $P^{2}_{6}\subseteq G$, in this case, $d(v)\geq13-9=4$ then $P^2_6\subseteq G$.  If $K_{5}\subseteq G$, the vertex $v\in V(G-K_{5})$ is adjacent to at most one vertex of the $K_{5}$, otherwise, $P^{2}_{6}\subseteq G$. Therefore, $e(G)\leq11<f_{P^{2}_{6}}(6)$.

When $n=5$, since $e(K_{5})=10>f_{P^{2}_{6}}(5)$, the statement is not true, then we show that the statement is true for $n=11$. If $P^{2}_{5}\nsubseteq G$, by Theorem \ref{P5},
$e(G)\leq\left\lfloor\frac{11^2+11}{4}\right\rfloor<f_{P^{2}_{6}}(11)$. If $P^{2}_{5}\subseteq G$, first suppose that the graph spanned by the vertices of the copy of $P^{2}_{5}$ have at most 8 edges. It can be checked that every triangle can be adjacent to at most 7 edges of the $P^{2}_{5}$, otherwise, $P^2_6\subseteq G$. When there exists a triangle as subgraph in $G-V(P^{2}_{5})$, we get $e(G)\leq 8+7+9+\ex(6,P^2_6)=36=f_{P^{2}_{6}}(6)$. If not, $e(G)\leq 8+18+9=35<f_{P^{2}_{6}}(6)$. If $K^{-}_{5}\subseteq G$ ($K_5$ minus an edge) then each vertex $v\in V(G-K^{-}_{5})$ is adjacent to at most 2 vertices of $K^{-}_{5}$. We get $e(G)\leq 9+12+\ex(6,P^2_6)=33<f_{P^{2}_{6}}(6)$. 
If $K_{5}\subseteq G$ then each vertex $v\in V(G-P^{2}_{5})$ is adjacent to at most one vertex of $K_{5}$. Altogether we have at most $10+6+\ex(6,P^2_6)=28$ edges. From the above, $e(G)\leq36=f_{P^{2}_{6}}(11)$.

Suppose (\ref{3}) holds for all $l\leq n-1$ $(l\neq5)$, the following proof is divided into 2 parts.\\
\textbf{Case 1}.\ \ If $T\subseteq G$, then each vertex $v\in V(G-T)$ is adjacent to at most 3 vertices of the copy of $T$, otherwise, $P^{2}_{6}\subseteq G$. The graph spanned by the vertices of the copy of $T$ cannot have more than $ex(6,P^2_6)=12$ edges. Since $G-T$ is an $(n-6)$-vertex $P^{2}_{6}$-free graph and $\ex(6,T)=12$, we have 
\begin{eqnarray}\label{P_6}
e(G)\leq12+3(n-6)+e(G-T)\leq3n-6+\ex(n-6,P^{2}_{6}).\end{eqnarray}
By the induction hypothesis,
$$ \ex(n-6,P^{2}_{6})\leq f_{P^{2}_{6}}(n-6)=\left\{
\begin{aligned}
&\left\lfloor\frac{(n-6)^{2}}{4}\right\rfloor+\left\lfloor\frac{n-7}{2}\right\rfloor,~n\equiv1,2,3~(\bmod~6), \\
&\left\lfloor\frac{(n-6)^{2}}{4}\right\rfloor+\left\lceil\frac{n-6}{2}\right\rceil,~\text{otherwise}.
\end{aligned}
\right.$$
We get
 $$\ex(n,P^{2}_{6})\leq\left\{
\begin{aligned}
&3n-6+\left\lfloor\frac{(n-6)^{2}}{4}\right\rfloor+\left\lfloor\frac{n-7}{2}\right\rfloor=\left\lfloor\frac{n^{2}}{4}\right\rfloor+\left\lfloor\frac{n-1}{2}\right\rfloor,~n\equiv1,2,3~(\bmod~6), \\
&3n-6+\left\lfloor\frac{(n-6)^{2}}{4}\right\rfloor+\left\lceil\frac{n-6}{2}\right\rceil=\left\lfloor\frac{n^{2}}{4}\right\rfloor+\left\lceil\frac{n}{2}\right\rceil,~\text{otherwise}.
\end{aligned}
\right.$$

\textbf{Case 2}.\ \ If $T\nsubseteq G$, by Theorem \ref{T}, $e(G)\leq \ex(n,T)\leq f_{P^{2}_{6}}(n)$ holds unless $n\equiv 8~(\bmod~12)$. When $n\equiv 8~(\bmod~12)$, then $e(G)\leq \ex(n,T)= f_{P^{2}_{6}}(n)+1$, however, by Theorem \ref{ET}, the equality holds only if $G$ is $T^{\frac{n}{2}}_{n}$, but $P^{2}_{6}\subseteq T^{\frac{n}{2}}_{n}~ (n\geq8)$, which implies that $e(G)\leq \ex(n,T)-1=f_{P^{2}_{6}}(n)$.\\
Summarizing, we obtain  \begin{displaymath}
\ex(n,P^2_6)= f_{P^{2}_{6}}(n)=\left\{
\begin{aligned}
&\left\lfloor\frac{n^{2}}{4}\right\rfloor+\left\lfloor\frac{n-1}{2}\right\rfloor,~n\equiv 1,2,3~(\bmod~6), \\
&\left\lfloor\frac{n^{2}}{4}\right\rfloor+\left\lceil\frac{n}{2}\right\rceil,~\otherwise. 
\end{aligned}
\right. 
\end{displaymath}
\end{proof}
\begin{proof}[Proof of Theorem \ref{EP6}]
It is obvious that
\begin{eqnarray}
\ex(n,T)\leq \ex(n,P^2_6),~\mathrm{except~when}~n\equiv 8~(\bmod~ 12).\end{eqnarray}
with strict inequality only when
\begin{eqnarray}
n \equiv~5,6,7,\mathrm{or}~11~(\bmod~12).\end{eqnarray}
We want to determine the graphs $G$ containing no copy of $P_6^2$ as a subgraph and satisfying $e(G)={\rm ex}(n, P_6^2)$.
Therefore suppose that $G$ possesses these properties. We claim that $G$ either contains a copy of $T$ as a subgraph
or it is either $F_n^{\lceil{n\over 2}\rceil ,\lceil{n\over 2}\rceil}$ or $F_n^{{n\over 2}+1 ,{n\over 2}+1}.$
If $n$ belongs to the set of integers given in (17) then this is obvious, since we have a strict inequality in (16). On the other hand for the other values of $n$ (except $n\equiv 8 \pmod{12}$) we obtain  ${\rm ex}(n,P_6^2)= {\rm ex}(n,T)=e(G)$. Theorem 6 describes these graphs.
However $G$ cannot be $T_n^{\lceil{n\over 2}\rceil}$ or $T_n^{{n\over 2}+1}$, because these graphs contain $P_6^2$
as a subgraph if $n\geq 7$. (In the case of $n=6$ we had strict inequality in (16).) The other possibility by Theorem 6 is that $G=S_n^{\lceil{n\over 2}\rceil}=F_n^{\lceil{n\over 2}\rceil ,\lceil{n\over 2}\rceil}.$
In the exceptional case we can use Proposition 11. According to this $G$ could be $T_n^{n\over 2}, S_n^{n\over 2}$ or
$S_n^{{n\over 2}+1}$. The first of them is excluded since $P_6^2\subset T_n^{n\over 2}$ the second and third ones can be written in the form $F_n^{{n\over 2},{n\over 2}}$ and $F_n^{{n\over 2}+1 ,{n\over 2}+1}.$

From now on we suppose that $e(G)= {\rm ex}(n,P_6^2)$, the graph $G$ contains a copy of $T$ and no copy of $P_6^2$ and prove by induction that $G$ is a graph given in the theorem.

Let us list some graphs $L$ (coming up in the forthcoming proofs) containing $P_6^2$ as a subgraph:

($\alpha$) $L$ is obtained by adding any edge to $T$ different from $\{ a,e\}, \{ d,c\}$ and $\{ b,f\}$ on Figure 4.

($\beta$) Add  the edges $\{ a,e\}, \{ d,c\}, \{b,f\}$ to $T$ resulting in $T^{\prime}$. The graph $L$ is obtained by adding a new vertex $u$ to $T^{\prime}$ which is adjacent to three vertices of $T^{\prime}$ different from the sets
$\{b,c,e\}$ and $\{a,d,f\}$.

($\gamma$) $L$ is obtained by adding two new adjacent vertices $u$ and $v$ to $T^{\prime}$, which are both adjacent to $b, c$ and $e$. Then e.g. the square of the path $\{ u,v,c,e,b,d\}$ is in $L$.

($\delta$) $L$ is obtained by adding 4 new vertices $u,v,w,x$, forming a complete graph, to $T^{\prime}$, all of them adjacent to $a,d$ and $f$. Then e.g. the square of the path $\{ a,u,v,w,x,d\}$ is in $L$.

($\epsilon$) $L$ consists of a complete graph on 5 vertices and a 6th vertex adjacent to two of them.

($\zeta$) The vertices of $L$ are $p_i (1\leq i\leq 4)$ and $q_j (1\leq j\leq 2)$ where $p_1, p_2, p_3, p_4$ span a path and all pairs $(p_i, q_j)$ are adjacent. Then the square of the path $\{ p_1, q_1, p_2, p_3 , q_2, p_4\}$ is $L$.

Let us start with the base cases. Let $n=6$ and suppose $T\subset G$. By ($\alpha $) only the edges $\{ a,e\}, \{ d,c\}$ and $\{ b,f\}$ can be added to $T$. To obtain  ${\rm ex}(6,P_6^2)=12$ edges all three of them should be added. The so obtained graph $T^{\prime}$ is really $H_6^3$.

Consider now the case $n=7$. It is clear that (\ref{P_6}) holds with equality only when the subgraph spanned by $T$ contains 12 edges and the vertex $u$ not in $T$ is adjacent with exactly 3 vertices of $T$. Hence the subgraph spanned by $T$ is really $T^{\prime}$. By ($\beta $) $u$ can be adjacent to either $b,c,e$ or $a,d,f$. In the first case $G=H_7^3$, in the second one $G=F_7^{4,1}$, as desired.

If $n=8$, $e(G)={\rm ex}(8,P_6^2)=19$ and the equality in (\ref{P_6}) implies, again, that $T$ must span $T^{\prime}$ and the remaining two vertices $u$ and $v$ are adjacent to exactly 3 vertices of $T^{\prime}$: either to the set $\{ b,c,e\}$ or to
$\{ a,d,f\}$ and $\{u,v\} $ is an edge. If both $u$ and $v$ are adjacent to $\{ b,c,e\}$ then ($\gamma$) leads to a contradiction. If one of $u$ and $v$ is adjacent to $\{ b,c,e\}$, the other one to $\{ a,d,f\}$, then $G=F_8^{4,1}$. Finally if both of them are adjacent to $\{ a,d,f\}$, then $G=F_8^{5,2}$.

Suppose now that $n=9$, when $e(G)={\rm ex}(9,P_6^2)=24$ and (\ref{P_6}) implies that the three vertices $u,v,w$ not in $T^{\prime}$ form a triangle and all three possess the properties mentioned in the previous case. If two of them are adjacent to $\{ b,c,e\}$ then ($\gamma$) gives the contradiction. If one of the them is adjacent to
$\{ b,c,e\}$, the two other ones are adjacent to $\{ a,d,f\}$, then $G=F_9^{5,2}$. Finally if all three are adjacent to $\{ a,d,f\}$, then $G=H_9^{6}$.

The case $n=10$ and $e(G)={\rm ex}(10,P_6^2)=30$ is very similar to the previous ones. If one of the new vertices,
$u,v,w,x$ is adjacent to $\{ b,c,e\}$ and the other 3 are adjacent to $\{ a,d,f\}$, then $G=H_{10}^{6}$. Here it cannot happen, by ($\delta$), that all 4 are adjacent to $\{ a,d,f\}$.

Finally let $n=11$ where $e(G)={\rm ex}(11,P_6^2)=36$. This case is different from the previous ones, since we cannot
have all the potential edges (12 in the graph spanned by $T$, 10 among the other 5 vertices $u,v,w,x,y$, and 15 between the two parts) one is missing. We distinguish 3 cases according the place of the missing edge.

\emph{\textbf{(i)}} $T^{\prime}\subset G$, $\{ u,v,w,x,y\}$ spans a copy of $K_5$, but there are only 14 edges between the two parts. Then $T^{\prime}$ has one vertex $z\in \{ a,b,c,d,e,f\}$ incident to at least two of the 14 edges. Then ($\epsilon$)
leads to a contradiction.

\emph{\textbf{(ii)}} $T^{\prime}\subset G$,  $\{ u,v,w,x,y\}$ spans a copy of $K_5$ minus one edge, say $\{ x,y\}$, and all 15 edges between the two parts are in $G$.

If two adjacent vertices from the set $\{ u,v,w,x,y\}$ are both adjacent to  $\{ b,c,e\}$ then ($\gamma$) gives the contradiction. Therefore if $x$ is adjacent to $\{ b,c,e\}$ then
$u,v$ and $w$ must be adjacent to $\{ a,d,f\}$. If $y$ is also adjacent to $\{ a,d,f\}$ then we have 4 vertices spanning a $K_4$ and all adjacent to $\{ a,d,f\}$. Then we obtain a contradiction by ($\delta$). Otherwise $y$ is adjacent to $\{ b,c,e\}$ and $G=H_{11}^6$.

Suppose now that $x$ is adjacent to $\{ a,d,f\}$. If $u,v,w$ are all adjacent to $\{ a,d,f\}$ then ($\delta$) leads to a contradiction. Hence at least one of them, say $u$ is adjacent to $\{ b,c,e\}$. But ($\gamma$) implies that two adjacent ones from from the set $\{ u,v,w,x,y\}$ cannot be adjacent to $\{ b,c,e\}$. Hence $v,w,x,y$ are all adjacent to $\{ a,d,f\}$ giving a contradiction again, by ($\delta$).

\emph{\textbf{(iii)}} $T$ spans only 11 edges, $\{ u,v,w,x,y\}$ determines a $K_5$ and all 15 edges are connecting the two parts.
Then $T$ must have a vertex incident to two edges connecting $T$ with  $\{ u,v,w,x,y\}$. Here ($\epsilon$) gives a contradiction.

Now we are ready to start the inductional step. Suppose that the statement is true for $n-6$ where $n\geq 12$. Prove it for $n$. Let $e(G)={\rm ex}(n,P_6^2)$ and suppose that $T\subset G$. We have to prove that $G$ is of the form described in the theorem. By (\ref{P_6}) we know that the equality implies that $T$ must span the the subgraph $T^{\prime}$
with 12 edges, every vertex of $G^{\prime}=G-T^{\prime}$ is adjacent either to the vertices $b,c,e$ or the vertices $a,d,f$ and $G^{\prime}$ is an extremal graph for $n-6$. That is $G^{\prime}$ is one the following graphs: $F_n^{\lceil {n-6\over 2}\rceil,j}, F_n^{{n-6\over 2}+1,j}, H_n^{\lfloor {n-6\over 2}\rfloor}, H_n^{\lceil {n-6\over 2}\rceil}, H_n^{{\lceil {n-6\over 2}\rceil +1}}$. All these graphs have $n-6$ vertices, their vertex sets are divided into two parts, $X^{\prime}$ and $Y^{\prime}$ where $|X^{\prime}|$ is either $\lfloor {n-6\over 2}\rfloor$ or $\lceil {n-6\over 2}\rceil$ or $\lceil {n-6\over 2}\rceil +1$, there is a bipartite graph between $X^{\prime}$ and $Y^{\prime}$ and $X^{\prime}$ is covered by vertex-disjoint triangles and at most one star.

Color a vertex of $G^{\prime}$ by red if it is adjacent to the vertices $b,c,e$ and blue otherwise. By ($\gamma$) two red vertices cannot be adjacent. On the other hand 4 blue vertices cannot span a path by ($\zeta$). Suppose that there
is a red vertex in  $X^{\prime}$. Then all vertices of $Y^{\prime}$ are colored blue. (It is easy to check that
$n\geq 12$ implies $|Y^{\prime}|\geq 2$.) If there are two blue vertices also in $X^{\prime}$ then they span a path of length 4 that is a contradiction. We can have one blue vertex in $X^{\prime}$ only when it contains no triangle and the center $s$ of the star is blue, the other vertices are all red. This is called the first coloring. It is easy to see that the choice $X=\{ b,c,e,s\} \cup Y^{\prime}, Y=\{ a,d,f\} \cup (X^{\prime}-\{ s\})$ defines a graph possessing the properties of the expected extremal graphs: $X$ and $Y$ span a complete bipartite graph, there are no edges within $Y$, and $X$ is covered by one triangle and one star which are vertex disjoint.

The other case is when all vertices of $X^{\prime}$ are blue. In this case no vertex of $Y^{\prime}$ can be blue, otherwise this vertex and the 3 vertices of a triangle or the center of the star with two other vertices would span a path of length 4. That is all vertices of $Y^{\prime}$ are red. This is the second coloring. Then  the choice $X=\{ b,c,e\} \cup X^{\prime}, Y=\{ a,d,f\} \cup Y^{\prime}$ defines a graph possessing the properties of the expected extremal graphs.

We have seen that $G$ has the expected structure in both cases. We only have to check the parameters. If $n\equiv 0,4,5 \pmod{6}$ then $X^{\prime}$ contains no star, the first coloring cannot occur, in the case of the second coloring 3-3 vertices are added to both parts, containing a triangle ($\{b,c,e\}$) in the $X$-part. The upper index increases by 3 in all cases when moving from $n-6$ to $n$.

Consider now the case $n\equiv 1 \pmod{6}$. If $G^{\prime}=H_{n-6}^{\lfloor {n-6\over 2}\rfloor }$ then we can proceed like in the previous cases, and $G=H_n^{\lfloor {n\over 2}\rfloor }$ is obtained. Suppose that $G^{\prime}=F_{n-6}^{\lceil {n-6\over 2}\rceil ,j}$. If $j<\lceil {n-6\over 2}\rceil$ then, again, the second coloring applies and we obtain $G=F_{n}^{\lceil {n\over 2}\rceil ,j}$. If, however,  $j=\lceil {n-6\over 2}\rceil$ then both colorings result in
 $G=F_{n}^{\lceil {n\over 2}\rceil , \lceil {n\over 2}\rceil-3}$. Let us recall that $G=F_{n}^{\lceil {n\over 2}\rceil , \lceil {n\over 2}\rceil}$ was obtained in the case when $T\not\subset G$.

 The cases $n\equiv 2,3 \pmod{6}$ can be checked similarly. 
\end{proof}
\section{Open problems}
The following paragraphs show why we think that Conjecture 1 is true.

\begin{lemma}\label{open} If the graph $G$ is obtained by adding a path of $r$ vertices to one of the classes of the complete bipartite graph $K_{n,n} (n\geq r)$ then $G$ contains the square of a path containing $\lfloor {3r\over 2}\rfloor +1$
vertices.
\end{lemma}
\begin{proof}
Suppose first that $r=2s$ is even. Let $X$ and $Y$ be the two parts, where $|X|=|Y|=n$ all edges $\{x,y\} (x\in X, y\in Y)$ are in $G$. Moreover, $X$ contains the path $\{ x_1, x_2, \ldots , x_{2s} \}$. Then the square of the path $\{ y_1, x_1, x_2, y_2, x_3, x_4, y_3, \ldots , x_{2s -1}, x_{2s}, y_{s+1} \}$ is in $G$ for an arbitrary set of distinct vertices  $y_1, y_2, \ldots , y_{s+1} \in Y$. The number of vertices of this path is really $3s+1$.

If $k=2s+1$ is an odd number then the desired path is  $\{ y_1, x_1, x_2, y_2, x_3, x_4, y_3, \ldots, \break  x_{2s-1}, x_{2s}, y_{s+1}, x_{2s+1} \}$. 
\begin{figure}[ht]
    \centering
\begin{tikzpicture}

\draw (-3.2,0)  node[above]{$x_{1}$} (-2.2,0)  node[above]{$x_{2}$}  (-1.2,0) node[above]{$x_{3}$} (-0.2,0) node[above]{$x_4$} (.8,0) node[above]{$x_5$} (1.8,0) node[above]{$x_6$} ;

\draw (-2.5,-3) node[below]{$y_1$} (-1.2,-3)  node[below]{$y_2$}  (.3,-3) node[below]{$y_3$} (1.5,-3) node[below]{$y_4$} ;

\draw (-3.2,0) -- (-2.2,0) -- (-1.2,0)  (-1.2,0) -- (-0.2,0) -- (-.8,0) (-0.2,0) --(.8,0) -- (1.8,0) ;

\draw [thick,red] (-2.5,-3)--(-3.2,0) -- (-2.2,0)   (-2.2,0)--(-1.2,-3) --(-1.2,0)  (-1.2,0)--(-0.2,0)--(.3,-3) (.3,-3)--(.8,0)--(1.8,0) (1.8,0)--(1.5,-3);

\draw [thick,blue] (-2.5,-3)--(-2.2,0)   (-0.2,0)--(-1.2,-3)--(-3.2,0)  (-1.2,0)--(.3,-3)--(1.8,0) (1.5,-3)--(.8,0);

\draw  (2.3,-3) arc (0:360:3cm and .85cm);

\draw (3.3,0) node[left]{$P_6$};

\draw (2.3,0) arc (0:360:3.0cm and 0.85cm);

\foreach \x in {-3.2,-2.2,-1.2,-0.2,.8,1.8}{

\filldraw (\x,0) circle (2pt) ;
}

\foreach \x in {-2.5,-1.2,0.3,1.5}{

\filldraw  (\x,-3) circle (2pt);
}

\end{tikzpicture}\qquad\qquad 
\begin{tikzpicture}

\draw  (0,0) --  (.5,1) -- (1,0) (.5,1) -- (1.5,1)  (1,0) --  (2,0) -- (2.5,1) (2,0) -- (1.5,1) (2,0)--(3,0)--(2.5,1) (2.5,1)--(3.5,1)--(3,0) (3.5,1)--(4,0)--(3,0) (3.5,1)--(4.5,1)--(4,0);

\filldraw (0,0) circle (2pt) -- (1,0) circle (2pt) --  (1.5,1) circle (2pt) -- (2.5,1) circle (2pt);
\filldraw (.5,1) circle (2pt)  (2,0) circle (2pt) (3,0) circle (2pt) (3.5,1) circle (2pt) (4,0) circle (2pt) (4.5,1) circle (2pt) ;

\draw (0,0) node[below]{$y_1$}  (1,0) node[below]{$x_2$}  (2,0) node[below]{$x_3$} (3,0) node[below]{$y_3$} (4,0) node[below]{$x_6$} (1.5,1) node[above]{$y_2$}  (2.5,1) node[above]{$x_4$}  (0.5,1) node[above]{$x_1$}  (3.5,1) node[above]{$x_5$} (4.5,1) node[above]{$y_4$}
(2.2,-1.5) node{$P^2_{10}$};
\end{tikzpicture}
    \caption{}
\end{figure}
\end{proof} 
It is easy to see, on the basis of Lemma \ref{open} that if this graph does not contain $P_k^2$ then $X$ cannot contain
a path of length $\lfloor {2k\over 3}\rfloor $. Now the obvious question is that at most how many edges can be chosen
in $X$ without having a path of given length. 
As one of the earliest results in extremal Graph Theory Erd\H{o}s and Gallai \cite{EG} proved the following result on the extremal number of paths.    

\begin{theorem} [\textbf{Erd\H os and Gallai}\cite{EG}] The maximum number of edges in an $n$-vertex $P_l$-free graph is ${n(l-2)\over 2}$, that is $\ex (n,P_l)\leq {n(l-2)\over 2}$ with equality if and only if $(l-1)|n$ and the graph is a vertex disjoint union of ${n\over l-1}$ complete graphs on $l-1$ vertices.
\end{theorem}


Faudree and Schelp\cite{FS} and independently Kopylov \cite{K} improved this result determining $\ex(n,P_{l})$ for every $n>l>0$ as well as the corresponding extremal graphs.

\begin{theorem} [\textbf{Faudree and Schelp}\cite{FS} \textbf{and independently Kopylov} \cite{K}] \label{exact}
Let $n \equiv r~\pmod {l-1}$, $0 \leq r \leq l-1, l\geq 2$.~Then~
$$\ex(n,P_l)=\frac{1}{2}(l-2)n-\frac{1}{2}r(l-1-r).$$
\end{theorem}

Faudree and Schelp also described the extremal graphs which are either

(a) vertex disjoint union of $m$ ($n=m(l-1)+r$) complete graphs $K_{l-1}$ and a $K_r$ or 

(b) $l$ is even and $r=\frac{l}{2}$ or $\frac{l}{2}-1$ then another extremal graph can be obtained by taking a vertex disjoint union if $t$ copies of $K_{l-1}~ (0\leq t \leq m)$ and a copy of $K_{\frac{l}{2}-1}+\overline{K}_{n-(t+\frac{1}{2})(l-1)+\frac{1}{2}}$.
Where $\overline{G}$ denotes the edge complement of the graph $G$, and $G+H$ is defined as the graph obtained from the vertex disjoint union of $G$ and $H$ together with all edges between $G$ and $H$.

We believe that the extremal graph for $\ex(n,P^2_k)$ is a complete bipartite graph plus one of the constructions above in the larger class. Check now the cases solved.

If $k=4$, by lemma 12 we cannot have a path of length 2 (that is an edge) in one side.

If $k=5$ then $l=3$, a path of length 3 is forbidden in one side. According to statements above we can have only vertex disjoint edges.

If $k=6$ then $l=4$ and a path of length 4 is forbidden in one side. Now the extremal constructions for $P_{l}$ are either $(a)$ triangles plus eventually one edge or $(b)$ $t$ triangles plus a star with $n-3t$ vertices.

These are in accordance with our results. Note that in the case of $k=7$, the value $l=4$ obtained again. The expected maximum value is the same as in the case of $k=6$, but the assumptions are weaker!


\begin{thebibliography}{99}

\bibitem{DIR}
G. Dirac: Extensions of Tur\'{a}n's theorem on graphs, Acta Math. Acad. Sci. Hungar., 14, 417--422 (1963).

\bibitem{EG}
P. Erd\H{o}s and T. Gallai. On maximal paths and circuits of graphs. Acta Math. Acad. sci. Hungar., 10, 337--356 (1959).

\bibitem{EDR1}
P. Erd\H{o}s and A. H. Stone. On the structure of linear graphs. Bull. Amer. Math. Soc., 52, 1087--1091 (1946).

\bibitem{EDR2}
P. Erd\H{o}s and M. Simonovits. A limit theorem in graph theory. Studia Sci. Math. Hungar., 1, 51--57 (1966).

\bibitem{FS}
R. J. Faudree and R. H. Schelp. Path Ramsey numbers in multicolorings. J. Combin. Theory. Ser. B, 19, 150--160 (1975).

\bibitem{K}
G. N. Kopylov. Maximal paths and cycles in a graph. Dokl. Akad. Nauk SSSR, 234(1), 19--21 (1977). (English translation: Soviet Mathh. Dokl. 18(3), 593--596 (1977)).

\bibitem{LIU}
H. Liu. Extremal graphs for blow-ups of cycles and trees. The Electron. J. of combin. 20(1) (2013), P65.

\bibitem{MAN}
W. Mantel: Problem 28, soln. by H. Gouventak, W. Mantel, J. Teixeira de Mattes, F. Schuh and W.A. Wythoff. Wiskundige Opgaven 10, 60--61 (1907).
\end{thebibliography}
\end{document}